\documentclass[10pt]{article}
\usepackage[dvips]{graphicx}
\usepackage{amssymb,color,latexsym}
\usepackage[centertags]{amsmath}
\usepackage{amsfonts}
\usepackage[amsmath,thmmarks]{ntheorem}
\usepackage{comment}
\usepackage{multirow}

\def\BBox{\kern  -0.2cm\hbox{\vrule width 0.2cm height 0.2cm}}

\def\subjclass#1{{\renewcommand{\thefootnote}{}%
\footnote{\emph{Mathematics Subject Classification (2010):} #1}}}
\def\keywords#1{\par\medskip
\noindent\textbf{Keywords.} #1}

\newtheorem{theorem}{Theorem}
\newtheorem{lemma}{Lemma}

\newtheorem{corollary}{Corollary}
\newtheorem{proposition}{Proposition}

\newcommand{\poo}{\mathcal{P}}
\newcommand{\po}{\mathcal{M}}
\newcommand{\fl}{\mathcal{F}}
\newcommand{\bl}{\mathcal{B}}
\newcommand{\Aut}{\mathrm{Aut}}
\newcommand{\gr}{\mathcal{G_M}}
\newcommand{\grr}{\mathcal{G_P}}
\newcommand{\Gpp}{\mathcal{G_P}}

\newcommand{\Mon}{\mathrm{Mon}}
\newcommand{\ma}{\mathcal{M}}
\newcommand{\W}{\mathcal{W}}

\title{Symmetry Type Graphs of Polytopes and Maniplexes}
\author{Gabe Cunningham\footnote{gabriel.cunningham@gmail.com}\\
{\small University of Massachusetts Boston, USA};\\
Mar\'ia del R\'io-Francos\footnote{maria.delrio@fmf.uni-lj.si} \\
{\small Institute of Mathematics Physics and Mechanics}\\
	{\small  University of Ljubljana, Slovenia};\\
Isabel Hubard\footnote{isahubard@im.unam.mx}  \ and Micael Toledo\footnote{micael50@hotmail.com}\\
{\small Instituto de Matem\'{a}ticas} \\
	{\small  Universidad Nacional Aut\'{o}noma de M\'{e}xico, M\'{e}xico}}

\begin{document}

\maketitle

\begin{abstract}
A $k$-orbit maniplex is one that has $k$ orbits of flags under the action of its automorphism group.
In this paper we extend the notion of symmetry type graphs of maps to that of  maniplexes and polytopes
and make use of them to study $k$-orbit maniplexes, as well as fully-transitive 3-maniplexes. 
In particular, we show that there are no fully-transtive $k$-orbit 3-mainplexes with $k > 1$ an odd number, we classify 3-orbit mainplexes and determine all face transitivities for 3- and 4-orbit maniplexes.
Moreover, we give generators of the automorphism group of a polytope or a maniplex, given its symmetry type graph. 
Finally, we extend these notions to oriented polytopes, in particular we classify oriented 2-orbit maniplexes and give generators for their orientation preserving automorphism group.
\end{abstract}


\section{Introduction}
While abstract polytopes are a combinatorial generalisation of classical polyhedra and polytopes, maniplexes generalise maps on surfaces and (the flag graph of) abstract polytopes. 
The combinatorial structure of maniplexes, maps and polytopes is completely determined by a edge-coloured $n$-valent graph with chromatic index $n$, often called the flag graph.
The symmetry type graph of a map is the quotient of its flag graph under the action of the automorphism group.
In this paper we extend the notion of symmetry type graphs of maps to that of maniplexes (and polytopes).
Given a maniplex, its symmetry type graph encapsulates all the information of the local configuration of the flag orbits under the action of the automorphism group of the maniplex. 

Traditionally, the main focus of the study of maps and polytopes has been that of their symmetries.
Regular and chiral ones have been extensively studied. These are maps and polytopes with either maximum degree of symmetry or maximum degree of symmetry by rotation.
Edge-transitive maps were studied in \cite{edge-trans} by Siran, Tucker and Watkins. 
Such maps have either 1, 2 or 4 orbits of flags under the action of the automorphism group.
More recently Orbani\'c, Pellicer and Weiss extend this study and classify $k$-orbit maps (maps with $k$ orbits of flags under the automorphism group) up to $k\leq4$ in \cite{k-orbitM}.
 Little is known about polytopes that are neither regular nor chiral. In \cite{tesisisa} Hubard gives a complete characterisation of the automorphism group of 2-orbit and fully-transitive polyhedra (i.e. polyhedra transitive on vertices, edges and faces) in terms of distinguished generators of them. Moreover, she finds generators of the automorphism group of a 2-orbit polytope of any given rank.

Symmetry type graphs of the Platonic and Archimedean Solids were determined in \cite{Archim}.
In \cite{medial} Del R\'io-Francos, Hubard, Orbani\'c and Pisanski determine symmetry type graphs of up to 5 vertices and give, for up to 7 vertices, the possible symmetry type graphs that a properly self-dual, an improperly self-dual and a medial map might have. 
The possible symmetry type graphs that a truncation of a map can have is determined in \cite{trunc}.
One can find in \cite{CompSymTypeGraph}  a strategy to generate symmetry type graphs.

By making use of symmetry type graphs, in this paper we classify 3-orbit polytopes and
give generators of their automorphism groups.
In particular, we show that 3-orbit polytopes are never fully-transitive, but they are $i$-face-transitive for all $i$ but one or two, depending on the class.
We extend further the study of symmetry type graphs to show that if a 4-orbit polytope is not fully-transitive, then it is $i$-face-transitive for all $i$ but at most three ranks.
Moreover, we show that a fully-transitive 3-maniplex (or 4-polytope) that is not regular cannot have an odd number of orbits of flags, under the action of the automorphism group.

The main result of the paper is stated in Theorem~\ref{auto}. Given a maniplex $\ma$ in Theorem~\ref{auto} we give generators for the automorphism group of $\ma$ with respect to some base flag.

The paper is divided into six sections, organised in the following way. 
In Section \ref{sec:PolyMani},
 we review some basic theory of polytopes and maniplexes, and  describe their respective flag graphs. 
In Section \ref{sec:stg}, 
we define and give some properties of the symmetry type graphs of polytopes and maniplexes, extending the concept of symmetry type graphs of maps.
In Section \ref{sec:stg-highly},
we study symmetry type graphs of highly symmetric maniplexes. In particular, we
classify symmetry type graphs with 3 vertices, determine the possible transitivities that a 4-orbit mainplex can have and study some properties of fully-transitive maniplexes of rank 3.
In Section \ref{Gen-autG}
we give  generators of the automorphism group of a polytope or a maniplex. In the last section of the paper we define oriented and orientable maniplexes. Further on, we define the oriented flag di-graph which emerge from a flag graph if this is bipartite. 
The oriented symmetry type di-graph of an oriented maniplex is then a quotient of the oriented flag di-graph, just as the symmetry type graph was a quotient of the flag graph.
Using these graphs we classify oriented 2-orbit maniplexes and give generators for their orientation preserving automorphism group.

\section{Abstract Polytopes and Maniplexes}
\label{sec:PolyMani}

\subsection{Abstract Polytopes}

In this section we briefly review the basic theory of abstract polytopes and their monodromy groups (for details we refer the reader to \cite{arp} and \cite{d-auto}). 

An (\emph{abstract\/}) \emph{polytope of rank\/} $n$, or simply an \emph{$n$-polytope\/}, is a partially ordered set $\mathcal{P}$ with a strictly monotone rank function with range $\{-1,0, \ldots, n\}$. 
An element of rank $j$ is called a \emph{$j$-face\/} of $\mathcal{P}$, and a face of rank $0$, $1$ or $n-1$ is called a \emph{vertex\/}, \emph{edge\/} or \emph{facet\/}, respectively. 
A {\em chain of $\poo$} is a totally ordered subset of $\poo$. 
The maximal chains, or \emph{flags}, all contain exactly $n + 2$ faces, including a unique least face $F_{-1}$ (of rank $-1$) and a unique greatest face $F_n$ (of rank $n$).  
A polytope $\mathcal{P}$ has the following homogeneity property (diamond condition):\ whenever $F \leq G$, with $F$ a $(j-1)$-face and $G$ a $(j+1)$-face for some $j$, then there are exactly two $j$-faces $H$ with $F \leq H \leq G$. Two flags are said to be \emph{adjacent} ($i$-\emph{adjacent}) if they differ in a single face (just their $i$-face, respectively).
The diamond condition can be rephrased by saying that every flag $\Phi$ of $\poo$ has a unique $i$-adjacent flag, denoted $\Phi^i$, for each $i=0, \dots, n-1$. Finally, $\mathcal{P}$ is  \emph{strongly flag-connected}, in the sense that, if $\Phi$ and $\Psi$ are two flags, then they can be joined by a sequence of successively adjacent flags, each containing $\Phi \cap \Psi$.

Let $\mathcal{P}$ be an abstract $n$-polytope. 
The {\em universal} string Coxeter group $W := [\infty,\ldots,\infty]$ of rank $n$, with distinguished involutory generators $r_0,$ $r_1,\ldots,r_{n-1}$, acts transitively on the set of flags $\fl(\poo)$ of $\poo$ in such a way that $\Psi^{r_i}  = \Psi^i$, the $i$-adjacent flag of $\Psi$, for each $i=0,\dots,n-1$ and each $\Psi$ in $\fl(\poo)$. In particular, if $w=r_{i_1}\ldots r_{i_k} \in W$ then 
\[ \Psi^w  =\Psi^{r_{i_{1}} r_{i_2}\ldots r_{i_{k-1}}r_{i_k}} 
  =: \Psi^{i_1,i_2,\ldots,i_{k-1},i_k}. \]
The {\em monodromy or connection group} of $\poo$ (see for example~\cite{d-auto}), denoted $\Mon(\poo)$, is the quotient of $W$ by the normal subgroup $K$ of $W$ consisting of those elements of $W$ that act trivially on $\fl(\poo)$ (that is, fix every flag of $\poo)$. Let 
\[ \pi: W \to \Mon(\poo)=W/K \] 
denote the canonical epimorphism. 
Clearly, $\Mon(\poo)$ acts on $\fl(\poo)$ in such a way that $\Psi^{\pi(w)}=\Psi^w$ for each $w$ in $W$ and each $\Psi$ in $\fl(\poo)$, so in particular $\Psi^{\pi(r_i)}=\Psi^{i}$ for each $i$. We slightly abuse notation and also let $r_i$ denote the $i$-th generator $\pi(r_i)$ of $\Mon(\poo)$. We shall refer to these $r_i$ as the {\em distinguished} generators of $\Mon(\poo)$.

Since the action of $W$ is transitive on the flags, 
the action of $\Mon(\poo)$ on the flags of $\poo$ is also transitive; moreover, this action is faithful, since only the trivial element of $\Mon(\poo)$ fixes every flag. Thus $\Mon(\poo)$ can be viewed as a subgroup of the symmetric group on $\fl(\poo)$. 
Note that for every flag $\Phi$ of $\poo$ and $i, j \in \{0, \dots, n-1\}$ such that $|i-j|\geq 2$, we have that $\Phi^{r_ir_j}=\Phi^{i,j}=\Phi^{j,i}=\Phi^{r_jr_i}$. Since the action of $\Mon(\poo)$ is faithful in $\fl(\poo)$, this implies that $r_ir_j=r_jr_i$, whenever $|i-j|\geq 2$.

An {\em automorphism} of a polytope $\poo$ is a bijection of $\poo$ that preserves the order. 
We shall denote by $\Aut(\poo)$ the group of automorphisms of $\poo$. 
Note that any automorphism of $\poo$ induces a bijection of its flags that preserves the $i$-adjacencies, for every $i \in \{0, 1,  \dots, n-1\}$. 
A polytope $\poo$ is said to be {\em regular} if the action of $\Aut(\poo)$ is regular on $\fl(\poo)$. 
If $\Aut(\poo)$ has exactly 2 orbits on $\fl(\poo)$ in such a way that adjacent flags belong to different orbits, $\poo$ is called a {\em chiral polytope}.
We say that a polytope is a {\em $k$-orbit polytope} if the action of $\Aut(\poo)$ has exactly $k$ orbits on $\fl(\poo)$. Hence, regular polytopes are 1-orbit polytopes and chiral polytopes are 2-orbit polytopes.

Given an $n$-polytope $\poo$, we define the {\em graph of flags} $\Gpp$ of $\poo$ as follows. The vertices of $\Gpp$ are the flags of $\poo$, and we put an edge between two of them whenever the corresponding flags are adjacent. Hence $\Gpp$ is $n$-valent (i.e. every vertex of $\Gpp$ has exactly $n$ incident edges; to reduce confusion we avoid the alternative terminology `$n$-regular'). Furthermore, we can colour the edges of $\Gpp$ with $n$ different colours as determined by the adjacencies of the flags of $\poo$. That is, an edge of $\Gpp$ has colour $i$, if the corresponding flags of $\poo$ are $i$-adjacent. In this way every vertex of $\Gpp$ has exactly one edge of each colour (see Figure \ref{fig:baricentic}).

\begin{figure}[htbp]
\begin{center}
\includegraphics[width=5cm]{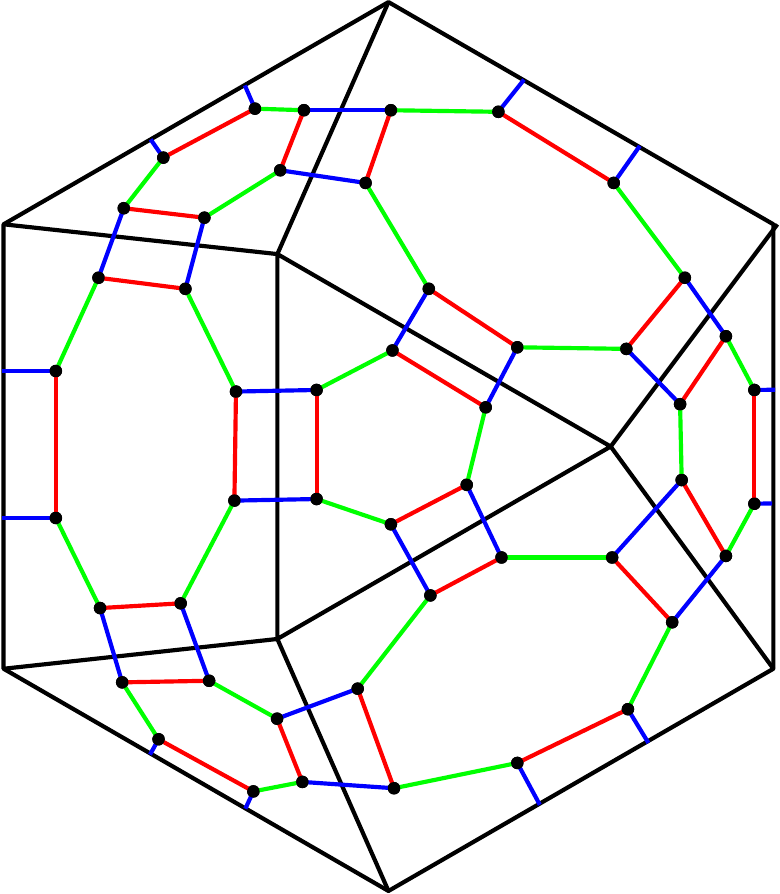}
\caption{The graph of flags of a cubeoctahedron}
\label{fig:baricentic}
\end{center}
\end{figure}

It is straightforward to see that each automorphism of $\poo$ induces an automorphism of the flag graph $\grr$ that preserves the colours. 
Conversely, every automorphism of $\grr$ that preserves the colours is a bijection of the flags that preserves all the adjacencies, inducing an automorphism of $\poo$. That is, the automorphism group $\Aut(\poo)$ of $\poo$ is the colour preserving automorphism group $\Aut_p(\grr)$ of $\grr$.

Note that the connectivity of $\poo$ implies that the action of $\Aut(\poo)$ on $\fl(\poo)$ is free (or semiregular). Hence, the action of $\Aut_p(\grr)$ is free on the vertices of the graph $\grr$.

One can re-label the edges of $\grr$ and assign to them the generators of $\Mon(\poo)$. In fact, since for each flag $\Phi$, the action of $r_i$ takes $\Phi$ to $\Phi^{r_i}$, by thinking of the edge of colour $i$ of $\grr$ as the generator $r_i$, one can regard a walk along the edges of $\grr$ as an element of $\Mon(\poo)$. That is, if $w$ is a walk along the edges of $\Mon(\poo)$ that starts at $\Phi$ and finishes at $\Psi$, then we have that $\Phi^w=\Psi$.
Hence, the connectivity of $\poo$ also implies that the action of $\Mon(\poo)$ is transitive on the vertices of $\grr$. 
Furthermore, since the $i$-faces of $\poo$ can be regarded as the orbits of flags under the action of the subgroup $H_i=\langle r_j \mid j\neq i \rangle$, the $i$-faces of $\poo$ can be also regarded as the connected components of the subgraph of $\grr$ obtained by deleting all the edges of colour $i$.

\subsection{Maniplexes}

Maniplexes were first introduced by Steve Wilson in \cite{mani}, aiming to unify the notion of maps and polytopes. In this section we review the basic theory of them.

An \emph{$n$-complex} $\mathcal{M}$ is defined by a set of {\em flags} $\fl$ and a sequence $(r_0,$ $ r_1, \dots, r_n)$, such that each $r_i$ partitions the set $\fl$ into sets of size 2 and the partitions defined by $r_i$ and $r_j$ are disjoint when $i\neq j$. Furthermore, we ask for $\po$ to be {\em connected} in the following way. Thinking of the $n$-complex $\mathcal{M}$ as the graph $\mathcal{G}$ with vertex set $\fl$, and with edges of colour $i$ corresponding to the matching $r_i$, we ask for the graph $\mathcal{G}$ indexed by $\po$ to be connected. 
An \emph{$n$-maniplex} is an $n$-complex such that the elements in the sequence $(r_0, r_1, \dots, r_n)$ correspond to the distinguished involutory generators of a Coxeter string group. 
In terms of the graph  $\mathcal{G}$, this means that the connected components of the induced subgraph with edges of colours $i$ and $j$, with $| i - j | \geq 2$ are 4-cyles.
We shall refer to the {\em rank} of an $n$-maniplex, precisely to $n$.

A 0-maniplex must be a graph with two vertices joined by an edge of colour 0. A 1-maniplex is associated to a 2-polytope or $l$-gon, which graph contains $2l$ vertices joined by a perfect matching of colours 0 and 1, and each of size $l$. A 2-maniplex can be considered as a map and vice versa, so that maniplexes generalise the notion of maps to higher rank. Regarding polytopes, the flag graph of any $(n+1)$-polytope can be associated to an $n$-maniplex, generalising in such way the notion of polytopes.

One can think of the sequence $(r_0, r_1, \dots, r_n)$ of a maniplex $\ma$ as permutations of the flags. In fact, if $\Phi, \Psi \in \fl$ are flags of $\ma$ belonging to the same part of the partition induced by $r_i$, for some $i$, we say that $\Phi^{r_i} = \Psi$ and $\Psi^{r_i} = \Phi$. In this way each $r_i$ acts as a involutory permutation of $\fl$.
In analogy with polytopes, we let 
$K = \{ w \in \langle r_0, \dots, r_n \rangle \mid \Phi^w = \Phi, \ \mathrm{for} \ \mathrm{all} \ \Phi \in \fl \}$ and define  the {\em connection group} $\Mon(\ma)$ of $\ma$ as the quotient of $\langle r_0, \dots, r_n \rangle$ over $K$. 
As before, we abuse notation and say that $\Mon(\ma)$ is generated by $r_0, \dots, r_n$ and define the action of $\Mon(\ma)$ on the flags inductively, induced by the action of the sequence $(r_0, r_1, \dots, r_n)$. In this way, the action of $\Mon(\ma)$ on $\fl$ is faithful and transitive.

Note further that since the sequence $(r_0, r_1, \dots, r_n)$ induces a string Coxeter group, then, as elements of $\Mon(\ma)$, $r_ir_j=r_jr_i$ whenever $|i-j|\geq 2$. This implies that given a flag $\Phi$ of $\po$ and $i, j \in \{0, \dots, n\}$ such that $|i-j|\geq 2$, we have that $\Phi^{i,j}=\Phi^{r_ir_j}=\Phi^{r_jr_i}=\Phi^{j,i}$. 

An {\em automorphism} $\alpha$ of an $n$-maniplex is a colour-preserving automorphism of the graph $\mathcal{G}$.
In a similar way as it happens for polytopes, the connectivity of the graph $\mathcal G$ implies that the action of the automorphism group $\Aut(\po$) of $\po$ is free on the vertices of $\mathcal G$.
Hence, $\alpha$ can be seen as a permutation of the flags in $\fl$ that commutes with each of the permutations in the connection group.

To have consistent concepts and notation between polytopes and maniplexes, we shall say that an $i$-face (or a face of rank $i$) of a maniplex is a connected component of the subgraph of $\mathcal G$ obtained by removing the $i$-edges of $\mathcal G$. Furthermore, we say that two flags $\Phi$ and $\Psi$ are $i$-adjacent if $\Phi^{r_i}= \Psi$ (note that since $r_i$ is an involution, $\Phi^{r_i}= \Psi$ implies that $\Psi^{r_i}= \Phi$, so the concept is symmetric).

To each $i$-face $F$ of $\po$, we can associate an $(i-1)$-maniplex $\po_F$ by identifying two flags of $F$
whenever there is a $j$-edge between them, with $j > i$. Equivalently, we can remove from $F$ all edges of colours $\{i+1, \ldots,
n-1 \}$, and then take one of the connected components. In fact, since $\langle r_0, \ldots, r_{i-1} \rangle$
commutes with $\langle r_{i+1}, \ldots, r_n \rangle$, the connected components of this subgraph of $F$ are
all isomorphic, so it does not matter which one we pick.

If $\Phi$ is a flag of $\po$ that contains the $i$-face $F$, then it naturally induces a flag $\overline{\Phi}$ in $\po_F$.
Similarly, if $\varphi \in \Aut(\po)$ fixes $F$, then $\varphi$ induces an automorphism $\overline{\varphi} \in \Aut(\po_F)$,
defined by $\overline{\Phi} \overline{\varphi} = \overline{\Phi \varphi}$. To check that this is well-defined,
suppose that $\overline{\Phi} = \overline{\Psi}$; we want to show that $\overline{\Phi \varphi} = \overline{\Psi \varphi}$.
Since $\overline{\Phi} = \overline{\Psi}$, it follows that $\Psi = \Phi^w$ for some $w \in \langle r_{i+1}, \ldots, r_n \rangle$.
Then $\Psi \varphi = (\Phi^w) \varphi = (\Phi \varphi)^w$, so that $\overline{\Psi \varphi} = \overline{\Phi \varphi}$.

By definition, the edges of $\mathcal G$ of one given colour form a perfect matching. The 2-factors of the graph $\mathcal G$ are the subgraphs spanned by the edges of two different colours of edges. 

Since the automorphisms of $\po$ preserve the adjacencies between the flags, it is not difficult to see that the following lemma holds. 

\begin{lemma}
\label{orbitTOorbit}
Let $\Phi$ be a flag of $\po$ and let $a \in \Mon(\po)$. If ${\mathcal O}_1$ and ${\mathcal O}_2$ denote the flag orbits of $\Phi$ and $\Phi^a$ (under $\Aut(\po)$), respectively, then $\Psi \in {\mathcal O}_1$  if and only if $\Psi^a \in {\mathcal O}_2$.
\end{lemma}

We say that a maniplex $\po$ is {\em $i$-face-transitive} if $\Aut(\po)$ is transitive on the faces of rank $i$. We say that $\po$ is {\em fully-transitive} if it is $i$-face-transitive for every $i =0, \dots, n-1$.

If $\Aut(\po)$ has $k$ orbits on the flags of $\po$, we say that $\po$ is a $k$-orbit maniplex. A 1-orbit maniplex is also called a {\em reflexible} maniplex. A 2-orbit maniplex with adjacent flags belonging to different orbits is a {\em chiral} maniplex. If a maniplex has at most 2 orbits of flags and $\gr$ is a bipartite graph, then the maniplex is said to be {\em rotary}.

\section{Symmetry type graphs of polytopes and maniplexes}
\label{sec:stg}

In this section we shall define the symmetry type graph of a polytope or a maniplex.  
To this end, we shall make use of quotient of graphs. 
Therefore, we now consider  pregraphs; that is, graphs that allow multiple edges and semi-edges. 
 As it should be clear, it makes no difference whether we consider an abstract $n$-polytope or an $(n-1)$-maniplex. Hence, though we will consider maniplexes throughout the paper, similar results will apply to polytopes.

Given an edge-coloured graph $\mathcal G$, and a partition $\mathcal B$ of its vertex set $V$,  the {\em coloured quotient with respect to $\mathcal B$},  $\mathcal G_{\mathcal B}$, is defined as the pregraph with vertex set $\mathcal B$, such that for any two vertices $B,C \in {\mathcal B}$, there is a dart of colour $a$ from $B$ to $C$ if and only if there exists $u \in B$ and $v \in C$ such that there is a dart of colour $a$ from $u$ to $v$. 
Edges between vertices in the same part of the partition $\mathcal B$ quotient into semi-edges.

Throughout the remainder of this section, let $\po$ be an  $(n-1)$-maniplex and $\gr$ its coloured flag graph. 

As we discussed in the previous section, $\Aut(\po)$ acts semiregularly on the vertices of $\gr$. We shall consider the orbits of the vertices of $\gr$ under the action of $\Aut(\po)$ as our partition ${\mathcal B}$, and denote ${\mathcal B:=O}rb$. 
Note that since the action is semiregular, every two orbits $B,C \in {\mathcal O}rb$ have the same number of elements. 
The {\em symmetry type graph} $T(\po)$ of $\po$ is the coloured quotient graph of $\gr$ with respect to ${\mathcal O}rb$. 

Since the flag graph $\gr$ is an undirected graph, then $T(\po)$ is a pre-graph without loops or directed edges. 
Furthermore, as we are taking the coloured quotient, and $\gr$ is edge-coloured with $n$ colours, then $T(\po)$ is an $n$-valent pre-graph, with one edge or semi-edge of each colour at each vertex. 
It is hence not difficult to see that if $\po$ is a reflexible maniplex, then $T(\po)$ is a graph consisting of only one vertex and $n$ semi-edges, all of them of different colours. 
In fact, the symmetry type graph of a $k$-orbit maniplex has precisely $k$ vertices. 
Figure~\ref{STGrank3A} shows the symmetry type graph of a reflexible 2-maniplex (on the left), and the symmetry type graph of the cuboctahedron: the quotient graph of the flag graph in Figure~\ref{fig:baricentic} with respect to the automorphism group of the cubocahedron. 

\begin{figure}[htbp]
\begin{center}
\includegraphics[width=10cm]{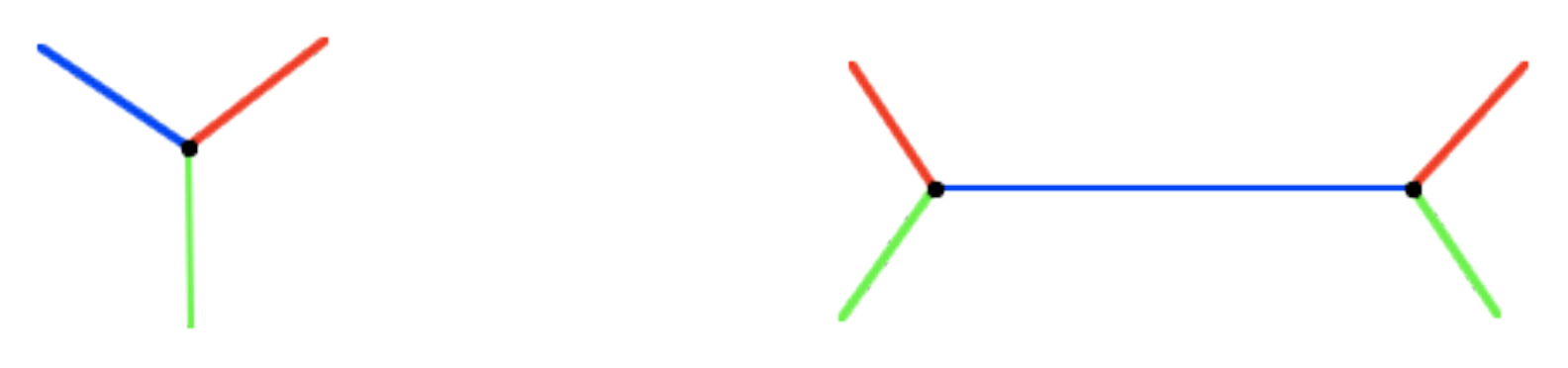}
\caption{Symmetry type graphs of a reflexible 2-maniplex (on the left) and of the cuboctahedron (on the right).}
\label{STGrank3A}
\end{center}
\end{figure}

Note that by the definition of $T(\po)$, there exists a surjective function $$\psi: V(\gr) \to V(T(\po))$$ that assigns, to each vertex  of $V(\gr)$ its corresponding orbit in $T(\po)$. Hence, given $\Phi, \Psi \in V(\gr)$, we have that $\psi(\Phi)=\psi(\Psi)$ if and only if $\Phi$ and $\Psi$ are in the same orbit under $\Aut(\po)$.

Given vertices $u,v$ of $T(\po)$, if there is an $i$-edge joining them, we shall denote such edge as $(u,v)_i$. Similarly, $(v,v)_i$ shall denote the semi-edge of colour $i$ incident to the vertex $v$. 

Because of Lemma~\ref{orbitTOorbit}, we can define the action of $\Mon(\po)$ on the vertices of $T(\po)$. In fact, given $v \in T(\po)$ and $a \in \Mon(\po)$, then $v^a:=\psi(\Phi^a)$, where $\Phi\in \psi^{-1}(v)$. Note that the definition of the action does not depend on the choice of  $\Phi\in \psi^{-1}(v)$; in fact, we have that $\Phi, \Psi, \in \psi^{-1}(v)$ if and only if $\psi(\Phi)=\psi(\Psi)$ and this in turn is true if and only if $\Phi$ and $\Psi$ are in the same orbit under $\Aut(\po)$. By Lemma~\ref{orbitTOorbit}, the fact that $\Phi$ and $\Psi$ are in the same orbit under $\Aut(\po)$ implies that, for any $a \in \Mon(\po)$, the flags $\Phi^a$ and $\Psi^a$ are also in the same orbit under $\Aut(\po)$. Hence $\psi(\Phi^a)=\psi(\Psi^a)$ and therefore the definition of $v^a$ does not depend on the choice of the element  $\Phi\in \psi^{-1}(v)$.

Since $\Mon(\po)$ is transitive on the vertices of $\gr$, then it is also transitive on the vertices of $T(\po)$, implying that $T(\po)$ is a connected graph. 
Furthermore, the action of each generator $r_i$ of $\Mon(\po)$ on a vertex $v$ of $T(\po)$ corresponds precisely to the (semi-)edge of colour $i$ incident to $v$. 
Hence, the orbit $v^{ \langle r_j \mid j \neq i \rangle}$ corresponds to the orbit under $\Aut(\po)$ of an $i$-face $F$ of $\po$ such that $F \in \Psi$, for some $\Psi \in \psi^{-1}(v)$ (as before, different choices of flag $\Psi\in \psi^{-1}(v)$ induce the same orbit of $i$-faces). Therefore, the  connected components of the subgraph $T^i(\po)$ of $T(\po)$ with edges of colours $\{0, \dots, n-1\} \setminus \{i\}$ correspond to the orbits of the $i$-faces under $\Aut(\po)$. In particular this implies the following proposition.

\begin{proposition}
Let $\po$ be a maniplex, $T(\po)$ its symmetry type graph and let $T^i(\po)$ be the subgraph resulting by erasing the $i$-edges of $T(\po)$. Then $\po$ is $i$-face-transitive if and only if $T^i(\po)$ is connected.
\end{proposition}

We shall say that a symmetry type graph $T$ is $i$-face-transitive if $T^i$ is connected, and that $T$ is a fully-transitive symmetry type graph if it is $i$-face-transitive for all $i$.

Recall that to each $i$-face $F$ of $\po$, there is an associated $(i-1)$-maniplex $\po_F$. 
The symmetry type graph $T(\po_F)$ is related in a natural way to the connected component of $T^i(\po)$ that corresponds to $F$:

\begin{proposition}
\label{STGofFaces}
Let $F$ be an $i$-face of the maniplex $\po$, and let $\po_F$ be the corresponding $(i-1)$-maniplex. Let
$\mathcal C$ be the connected component of $T^i(\po)$ corresponding to $F$. Then there is a surjective
function $\pi: V(\mathcal C) \to V(T(\po_F))$. Furthermore, if $j < i$ then each $j$-edge
$(u, u^j)_j$ of $\mathcal C$ yields a $j$-edge $(\pi(u), \pi(u^j))_j$ in $T(\po_F)$, and if $j > i$, then $\pi(u) = \pi(u^j)$.
\end{proposition}

\begin{proof}
First, let $\Phi$ and $\Psi$ be flags of $\po$ that are both in the connected component $F$, and suppose that
they lie in the same flag orbit, so that $\Psi = \Phi \varphi$ for some $\varphi \in \Aut(\po)$. 
Then the induced automorphism $\overline{\varphi}$ of $\po_F$ sends $\overline{\Phi}$ to $\overline{\Psi}$, and
therefore $\overline{\Phi}$ and $\overline{\Psi}$ lie in the same orbit. Furthermore, every flag of
$\po_F$ is of the form $\overline{\Phi}$ for some $\Phi$ in $F$. Thus, each orbit of $\po$ that
intersects $F$ induces an orbit of $\po_F$, and it follows that there is a surjective function
$\pi: V(\mathcal C) \to V(T(\po_F))$. 

Consider an edge $(u, u^j)_j$ in $\mathcal C$. Then $u = \psi(\Phi)$ for some flag $\Phi$ in $F$, and we can
take $u^j = \psi(\Phi^j)$. Both $\Phi$ and $\Phi^j$ induce flags in $\po_F$. If $j < i$, then
$\overline{\Phi^j} = \overline{\Phi}^j$. Therefore, there must be a $j$-edge from the orbit of
$\overline{\Phi}$ to the orbit of $\overline{\Phi^j}$; in other words, a $j$-edge from $\pi(u)$ to $\pi(u^j)$. 
On the other hand, if $j > i$, then $\overline{\Phi^j} = \overline{\Phi}$, and so $\overline{\Phi}$ and
$\overline{\Phi^j}$ lie in the same orbit and thus $\pi(u) = \pi(v)$.
\end{proof}

Note that the edges of a given colour $i$ of $T(\po)$ form a perfect matching (where, of course, we are allowing to match a vertex with itself by a semi-edge). Given two colours $i$ and $j$, the subgraph of $T(\po)$ consisting of all the vertices of $T(\po)$ and only the $i$- and $j$-edges shall be called a $(i,j)$ 2-factor of $T(\po)$.
Because $r_ir_j=r_jr_i$ whenever $| i-j |\geq 2$, in $\gr$, the alternating cycles of colours $i$ and $j$ have length 4. By Lemma~\ref{orbitTOorbit} each of these 4-cycles should then factor, in $T(\po)$, into one of the five graphs in Figure~\ref{4cyclequotient}. Hence, if $| i-j |\geq 2$, then the connected components of the $(i,j)$ 2-factors of $T(\po)$ are precisely one of these graphs. 

\begin{figure}[htbp]
\begin{center}
\includegraphics[width=8cm]{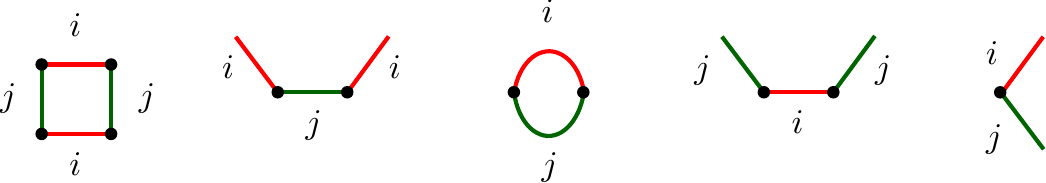}
\caption{Possible quotients of $i-j$ coloured 4-cycles.}
\label{4cyclequotient}
\end{center}
\end{figure}

In light of the above observations we state the following lemma.

\begin{lemma}
\label{2factors4vertices}
Let $T(\po)$ be the symmetry type graph of a maniplex. If there are vertices $u,v,w \in V(T(\po))$ such that $(u,v)_i, (v,w)_j \in E(T(\po))$ with $| i - j | \geq 2$, then the connected component of the $(i,j)$ 2-factor that contains $v$ has four vertices.
\end{lemma}

\section{Symmetry type graphs of highly symmetric maniplexes}
\label{sec:stg-highly}

One can classify maniplexes with small number of flag orbits (under the action of the automorphism group of the maniplex) in terms of their symmetry type graphs. The number of distinct possible symmetry types of a $k$-orbit $(n-1)$-maniplex is the number of connected pre-graphs on $k$ vertices that are $n$-valent and that can be edge-coloured with exactly $n$ colours. Furthermore, given a symmetry type graph, one can read from the appropriate coloured subgraphs the different types of face transitivities that the maniplex has.

As pointed out before, the symmetry type graph of a reflexible $(n-1)$-maniplex consists of one vertex and $n$ semi-edges. The classification of two-orbit maniplexes (see \cite{2-orbit}) in terms of the local configuration of their flags follows immediately from considering symmetry type graphs. In fact, for each $n$, there are $2^{n}-1$ symmetry type graphs with 2 vertices and $n$ (semi)-edges, since given any proper subset $I$ of the colours $\{0,1, \dots, n-1\}$, there is a symmetry type graph with two vertices, $| I |$ semi-edges corresponding to the colours of $I$, and where all the edges between the two vertices use the colours not in $I$ (see Figure~\ref{2orbitSTG}). 
This symmetry type graph corresponds precisely to polytopes in class $2_I$, see \cite{2-orbit}.

\begin{figure}[htbp]
\begin{center}
\includegraphics[width=8cm]{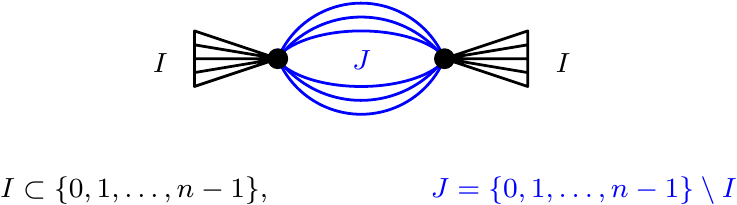}
\caption{The symmetry type graph of a maniplex in class $2_I$.}
\label{2orbitSTG}
\end{center}
\end{figure}

Highly symmetric maniplexes can be regarded as those with few flag orbits or those with many (or all) face transitivities. 
In \cite{medial} one can find the complete list of symmetry type graphs of $2$-maniplexes with at most 5 vertices. In this section we classify symmetry type graphs with 3 vertices and study some properties of symmetry type graphs of 4-orbit maniplexes and  fully-transitive 3-maniplexes.

\begin{proposition}\label{stg_3-orbit}
There are exactly $2n-3$ different possible symmetry type graphs of 3-orbit maniplexes of rank $n-1$.
\end{proposition}

\begin{proof}
Let $\ma$ be a 3-orbit $(n-1)$-maniplex and $T(\ma)$ its symmetry type graph. Then, $T(\ma)$ is an $n$-valent well edge-coloured graph with vertices $v_1, v_2$ and $ v_3$. Recall that the set of colours $\{0,1, \dots, n-1\}$ correspond to the distinguished generators $r_0, r_1, \dots, r_{n-1}$ of the connection group of $\ma$, and that by $(u,v)_i$ we mean the edge between vertices $u$ and $v$ of colour $i$.

Since $T(\ma)$ is a connected graph, without loss of generality, we can suppose that there is at least one edge joining $v_1$ and $v_2$ and another joining $v_2$ and $v_3$. 
Let $j,k \in \{0,1, \dots, n-1\}$ be the colours of these edges, respectively. That is, without loss of generality we may assume that $(v_1,v_2)_j$ and $(v_2,v_3)_k$ are edges of $T(\po)$. 
By Lemma~\ref{2factors4vertices}, we must have that $k = j \pm 1$, as otherwise $T(\po)$ would have to have at least 4 vertices.
This implies that the only edges of $T(\po)$ are either $(v_1,v_2)_j$ and $(v_2,v_3)_{j+1}$, $(v_1,v_2)_j$ and $(v_2,v_3)_{j-1}$ or $(v_1,v_2)_j$, $(v_2,v_3)_{j+1}$ and $(v_2,v_3)_{j-1}$, with $j \in \{1, 2, \dots, n-2\}$. (See Figure~\ref{i-transitive3}).

\begin{figure}[htbp]
\begin{center}
\includegraphics[width=12cm]{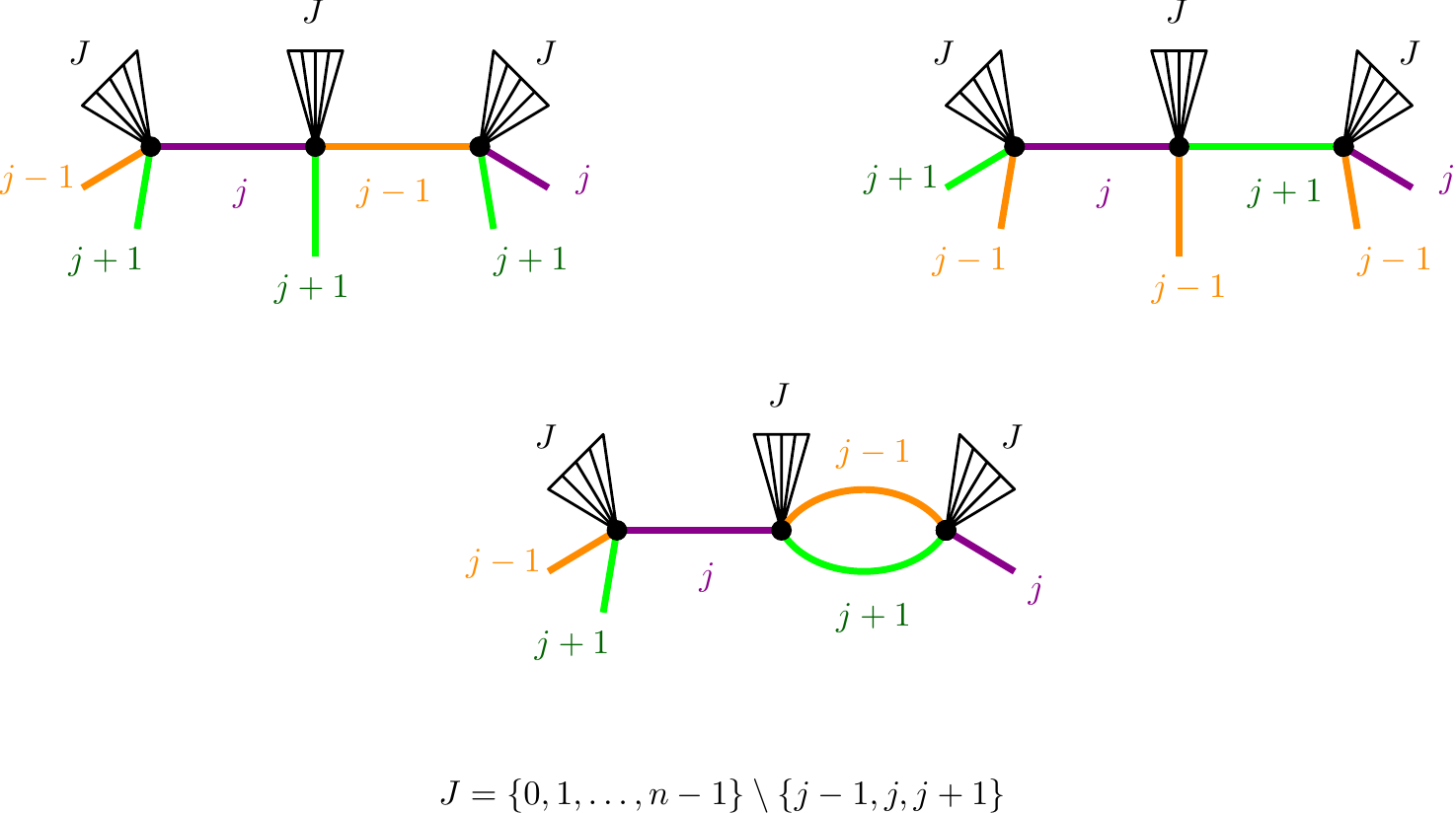}
\caption{Possible symmetry type graphs of 3-orbit $(n-1)$-maniplexes with edges of colours $j-1$, $j$, and $j+1$, with $j \in \{1, 2, \dots, n-2\}$.}
\label{i-transitive3}
\end{center}
\end{figure}
An easy computation now shows that there are $2n-3$ possible different symmetry type graphs of 3-orbit maniplexes of rank $n-1$.
\end{proof}

Given a 3-orbit $(n-1)$-maniplex $\po$ with symmetry type graph having exactly two edges $e$ and $e'$ of colours $j$ and $j+1$, respectively, for some $j \in \{0, \dots, n-2\}$, we shall say that $\po$ is in class $3^{j,j+1}$.
If, on the other hand, the symmetry type graph of $\po$ has one edge of colour $j$ and  parallel edges of colours $j-1$ and $j+1$, for some $j \in \{1, \dots, n-2\}$, then we say that $\po$ is in class $3^j$.
From Figure~\ref{i-transitive3} we observe that a maniplex in class $3^{j, j+1}$ is $i$-face-transitive whenever $i \neq j, j+1$, while a maniplex in class $3^j$ if $i$-face-transitive for every $i \neq j$.

\begin{proposition}
A $3$-orbit maniplex is $j$-face-transitive if and only if it does not belong to any of the classes $3^j$, $3^{j,j+1}$ or $3^{j-1,j}$.
\end{proposition}

\begin{theorem}
\label{no3-orbitfully}
There are no fully-transitive 3-orbit maniplexes.
\end{theorem}

Using Proposition~\ref{STGofFaces}, we get some information about the number of flag orbits that the $j$-faces have:

\begin{proposition}
A $3$-orbit maniplex in class $3^j$ or $3^{j,j+1}$ has reflexible $j$-faces.
\end{proposition}

\begin{proof}
If $\po$ is a $3$-orbit maniplex, then the orbits of the $j$-faces correspond to the connected components of 
$T^j(\po)$. Assuming that $\po$ is in class $3^j$ or $3^{j,j+1}$, the graph $T^j(\po)$ has two connected components;
an isolated vertex, and two vertices that are connected by a $(j+1)$-edge (and a $(j-1)$-edge, if $\po$ is in class
$3^{j,j+1}$. Then by Proposition~\ref{STGofFaces}, the $j$-faces that correspond to the isolated vertex
are reflexible (that is, 1-orbit), and the edge with label $j+1$ forces an identification between the two
vertices of the second component, so the $j$-faces in that component are also reflexible.
\end{proof}

\subsection{On the symmetry type graphs of 4-orbit maniplexes}
\label{sec:4notfully}

It does not take long to realise that counting the number of symmetry type graphs with $k\geq4$ vertices, and perhaps classifying them in a similar fashion as was done for 2 and 3 vertices,  becomes considerably more difficult.
In this section, we shall analyse symmetry type graphs with 4 vertices and determine how far a 4-orbit maniplex can be from being fully-transitive. The following lemma is a consequence of the fact that by taking away the $i$-edges of a symmetry type graph $T(\po$), the resulting $T^i(\po)$ cannot have too many components.

\begin{lemma}
\label{4orbMi-faceorb}
Let $\po$ be a 4-orbit $(n-1)$-maniplex and let $i \in \{0, \dots, n-1\}$. Then $\po$ has one, two or three orbits of $i$-faces.
\end{lemma}

If an $(n-1)$-maniplex $\po$ is not fully-transitive, there exists at least one $i \in \{0, \dots, n-1\}$ such that $T^i(\po)$ is disconnected. 
We shall divide the analysis of the types  in three parts: when $T^{i}(\po)$ has three connected components (two of them of one vertex and one with two vertices), when $T^i(\po)$ has a connected component with one vertex and another connected component with three vertices, and finally when $T^{i}(\po)$ has two connected components with two vertices each.
Before we start the case analysis, we let $v_1,v_2,v_3,v_4$ be the vertices of $T(\po)$.

Suppose that $T^{i}(\ma)$ has three connected components with $v_2$ and $v_3$ in the same component. 
Without loss of generality we may assume that $T(\po)$ has edges $(v_1,v_2)_{i}$ and $(v_3,v_4)_{i}$.
Let $k\in\{0,1, \dots, n-1\} \setminus \{i\}$ be the colour of an edge between $v_2$ and $v_3$. 
Since there is no edge of $T(\po)$ between $v_1$ and $v_4$, Lemma~\ref{2factors4vertices} implies that there are at most two such possible $k$, namely $k = i-1$ and $k = i+1$.
If $i \neq 0, n-1$, $T(\po)$ can have either both edges or exactly one of them, while if $i\in \{0, n-1\}$ there is one possible edge (see Figure~\ref{3components4}).

\begin{figure}[htbp]
\begin{center}
\includegraphics[width=10cm]{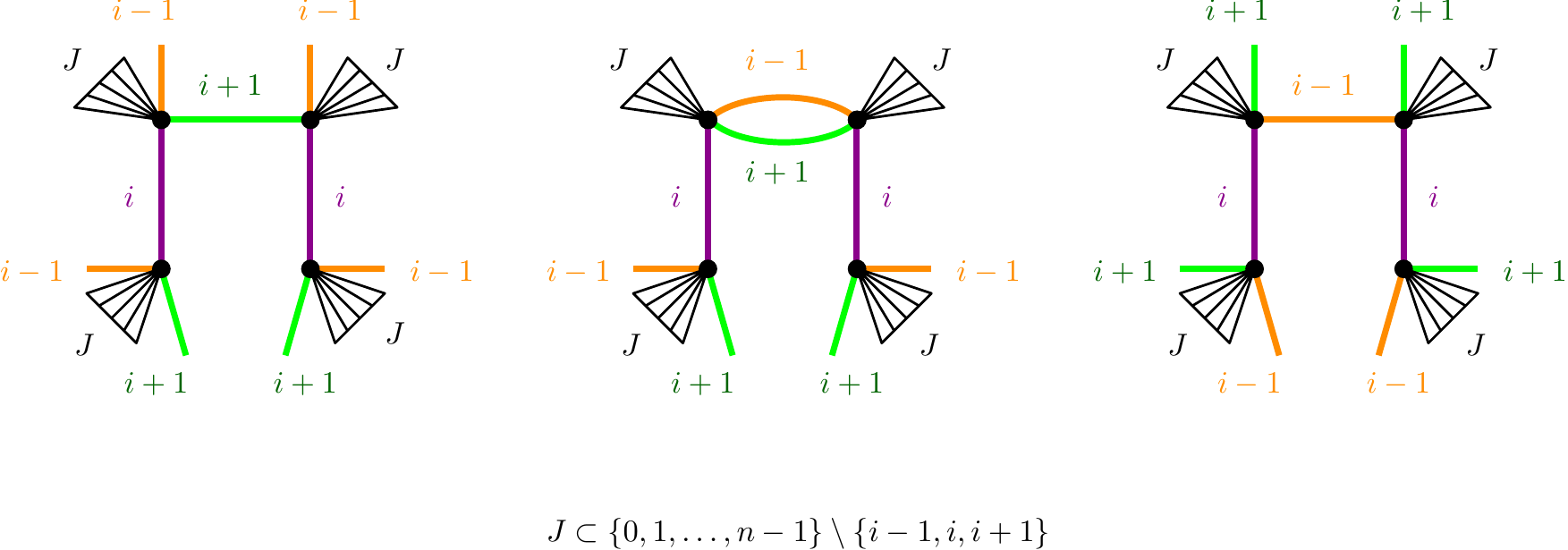}
\caption{Symmetry type graphs of an $(n-1)$-maniplex $\ma$ with four orbits on its flags, and three orbits on its $i$-faces.}
\label{3components4}
\end{center}
\end{figure}

Let us now assume that $T^{i}(\ma)$ has two connected components, one consisting of the vertex $v_1$ and the other one containing vertices $v_2, v_3$ and $v_4$. 
This means that the $i$-edge incident to $v_1$ is the unique edge that connects this vertex with the rest of the graph and, without loss of generality, $T(\po)$ has the edge $(v_1,v_2)_{i}$. 
As with the previous case, Lemma~\ref{2factors4vertices} implies that an edge between $v_2$ and $v_3$ has colour either $i-1$ or $i+1$.

First observe that having either  $(v_2,v_3)_{i-1}$ or $(v_2,v_3)_{i+1}$ in $T(\po)$ immediately implies (by Lemma~\ref{2factors4vertices}) that there is no edge between $v_2$ and $v_4$.
Now, if both edges $(v_2,v_3)_{i-1}$ and $(v_2,v_3)_{i+1}$ are in $T(\po)$, then an edge between $v_3$ and $v_4$ would have to have colour $i$, contradicting the fact that $T^i(\po)$ has two connected components.
Hence, there is exactly one edge between $v_2$ and $v_3$.
It is now straightforward to see that $T(\po)$ should be as one of the graphs in Figure~\ref{2components4}, implying that
 there are exactly four symmetry type graphs with these conditions for each $i\neq 0,1,n-2, n-1$, but only two symmetry type graph of this kind when $i=0, 1, n-2,$ or $n-1$.

\begin{figure}[htbp]
\begin{center}
\includegraphics[width=7cm]{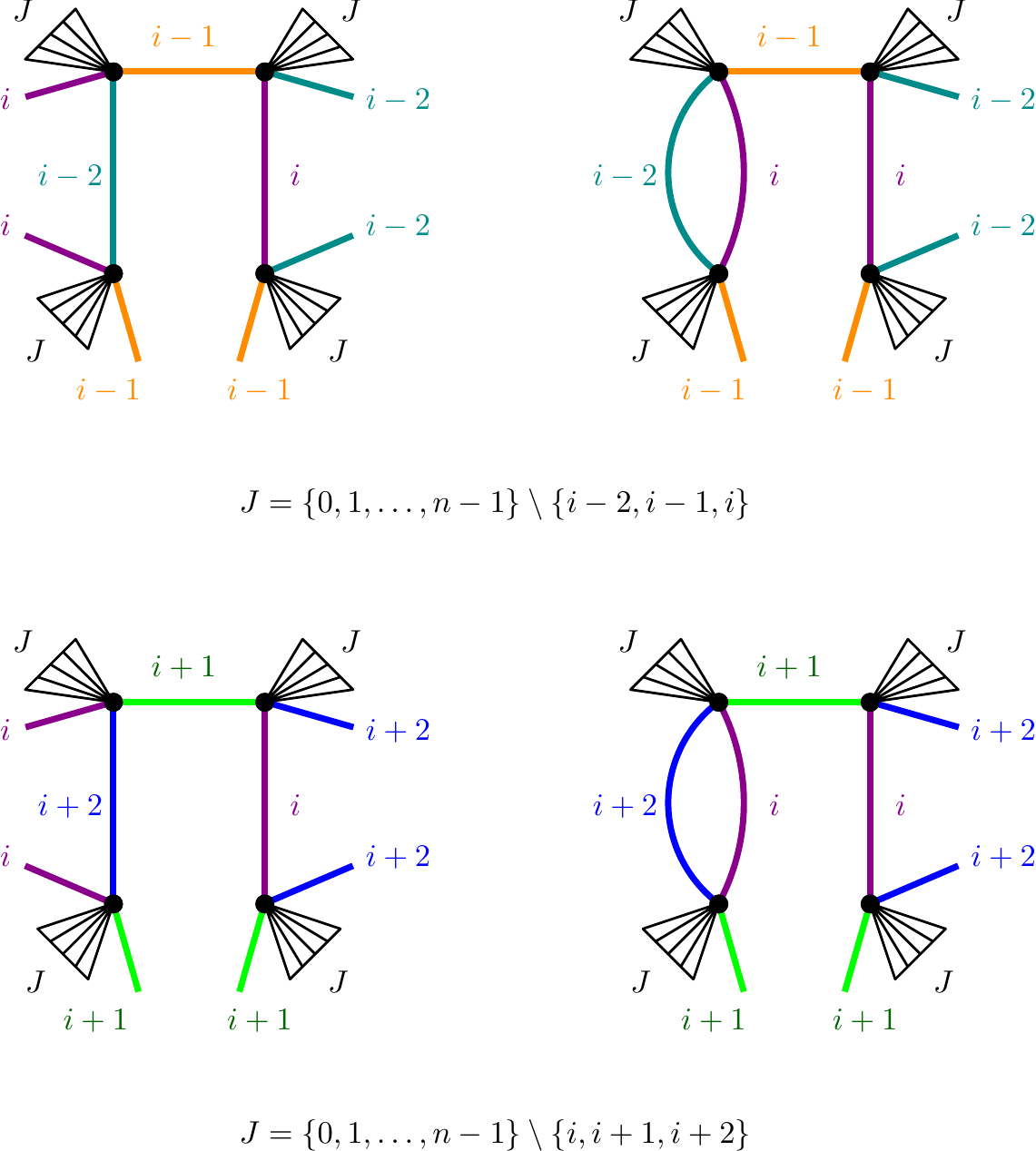}
\caption{Symmetry type graphs of $(n-1)$-maniplexes with four orbits on its flags, and two orbits on its $i$-faces such that one contains three flag orbits and the other contains a single flag orbit.}
\label{2components4}
\end{center}
\end{figure}

It is straightforward to see from Figure~\ref{2components4} that the next lemma follows.

\begin{lemma}
Let $\po $ be a 4-orbit $(n-1)$-maniplex with two orbits of $i$-faces such that $T^{i}(\ma)$ has a connected component consisting of one vertex, and another one consisting of three vertices. Then either $T^{i-1}(\ma)$ or $T^{i+1}(\ma)$ has two connected components, each with two vertices.
\end{lemma}

Finally, we turn out our attention to the case where $T^{i}(\ma)$ has two connected components, with two vertices each. Suppose that $v_1$ and $v_2$ belong to one component, while $v_3$ and $v_4$ belong to the other. 
As the two components must be connected by the edges of colour $i$, we may assume that $(v_1,v_3)_{i}$ is an edge of $T(\ma)$. 
If the vertices $v_2$ and $v_4$ have semi-edges of colour $i$, Lemma~\ref{2factors4vertices} implies that $T(\po)$ is one of the graphs  shown in Figure~\ref{2-2components4}.

\begin{figure}[htbp]
\begin{center}
\includegraphics[width=10cm]{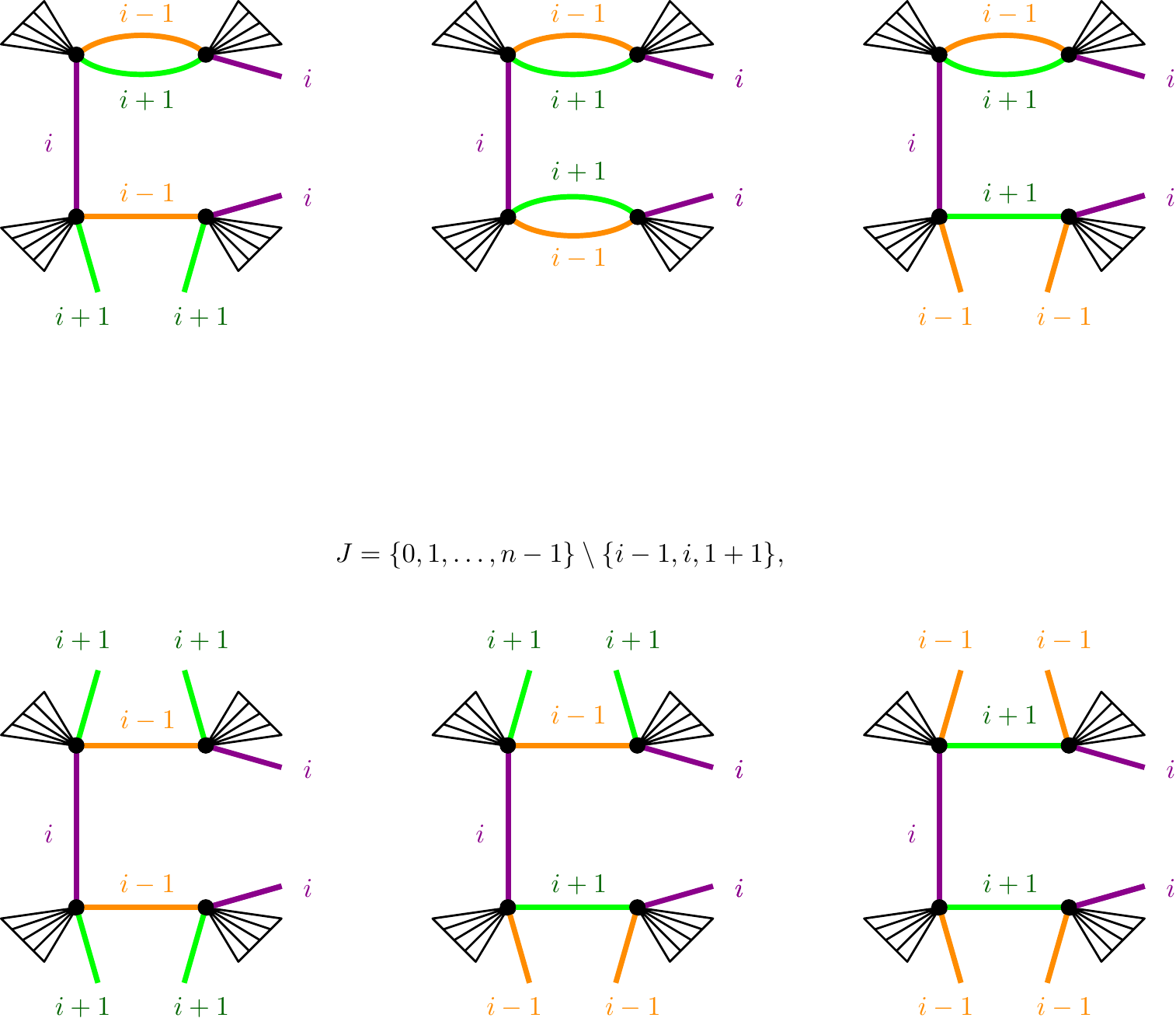}
\caption{Six of the symmetry type graphs of $(n-1)$-maniplexs with four orbits on its flags, and two orbits on its $i$-faces such that each contains two flag orbits.}
\label{2-2components4}
\end{center}
\end{figure}

On the other hand, if $(v_1,v_3)_{i}$ and $(v_2,v_4)_{i}$ are both edges of  $T(\ma)$, given $j \in \{0,1, \dots, n-1\} \setminus \{i-1, i, i+1\}$, we use again Lemma~\ref{2factors4vertices} to see that $(v_1,v_2)_j$ is an edge of $T(\po)$ if and only if $(v_3,v_4)_j$ is also an edge of $T(\po)$. 
By contrast, $T(\po)$ can have either two edges of colour $i\pm1$ (each joining the vertices of each connected component of $T^{i}(\po)$), four semi-edges or an edge and two semi-edges of colour $i\pm1$.
Hence, if $i \neq 0, n-1$, for each $J \subset \{0, 1, \dots, n-1\} \setminus \{i-1, i, i+1\}$ there are ten symmetry type graph with semi-edges of colours in $J$ and edges of colours not in $J$, as shown in Figures~\ref{3(n-2)_2compA} and~\ref{3(n-2)_2compB}, 
while for $J=\{0,1, \dots, n-1\} \setminus \{i-1, i, i+1\}$ there are six such graphs (shown in Figure~\ref{3(n-2)_2compB}). 
On the other hand if $i \in\{0, n-1\}$, for each $J \subset \{0, 1, \dots, n-1\} \setminus \{i-1, i, i+1\}$ there are two graphs as in Figure~\ref{3(n-2)_2compA} and one as in Figure~\ref{3(n-2)_2compB}, 
while for $J=\{0,1, \dots, n-1\} \setminus \{i-1, i, i+1\}$, there is only one of the graphs in Figure~\ref{3(n-2)_2compB}.

\begin{figure}[htbp]
\begin{center}
\includegraphics[width=7cm]{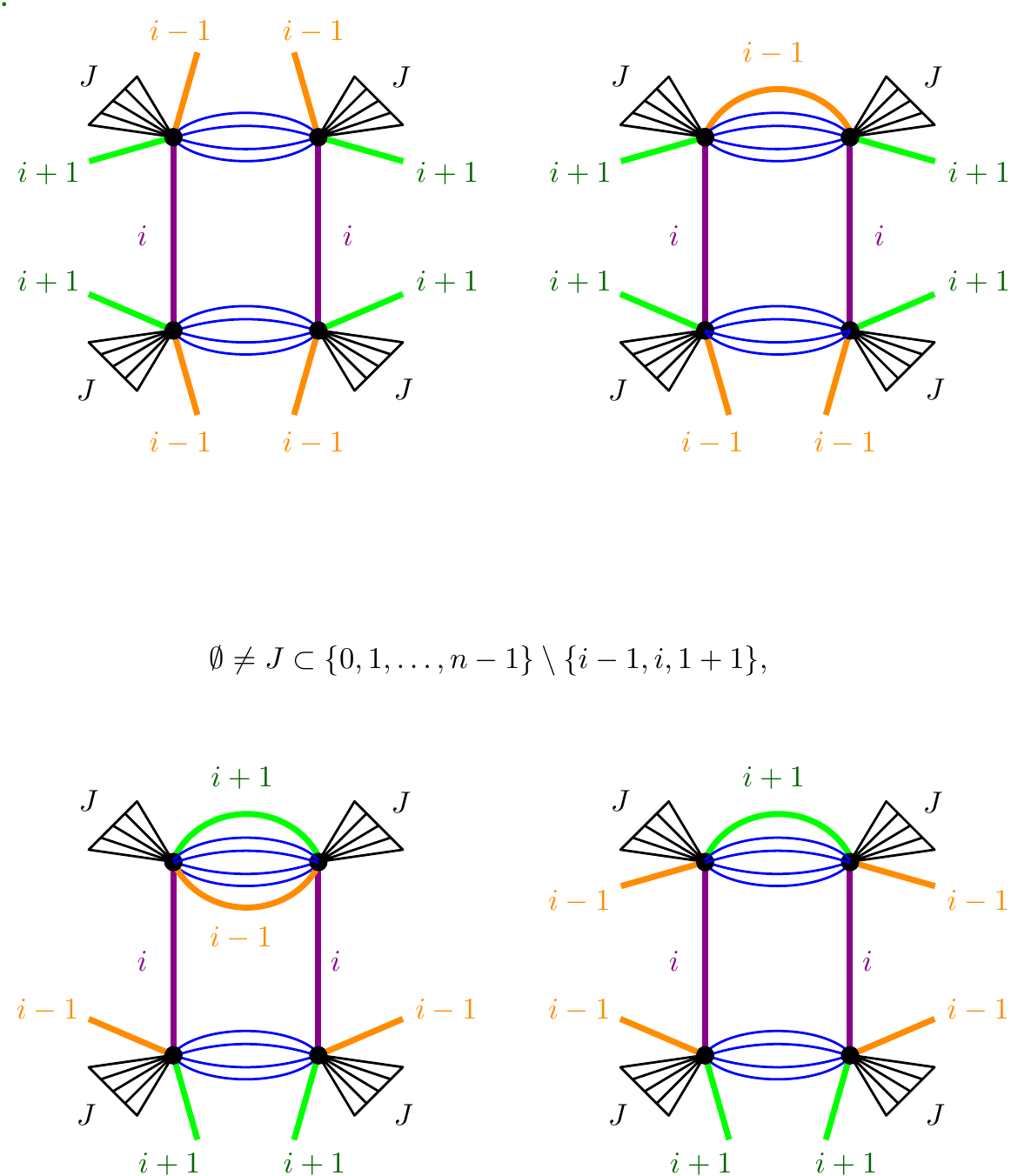}
\caption{Four families of possible symmetry type graphs of $(n-1)$-maniplexes with four orbits on its flags, and two orbits on its $i$-faces such that each contains two flag orbits.}
\label{3(n-2)_2compA}
\end{center}
\end{figure}

\begin{figure}[htbp]
\begin{center}
\includegraphics[width=12cm]{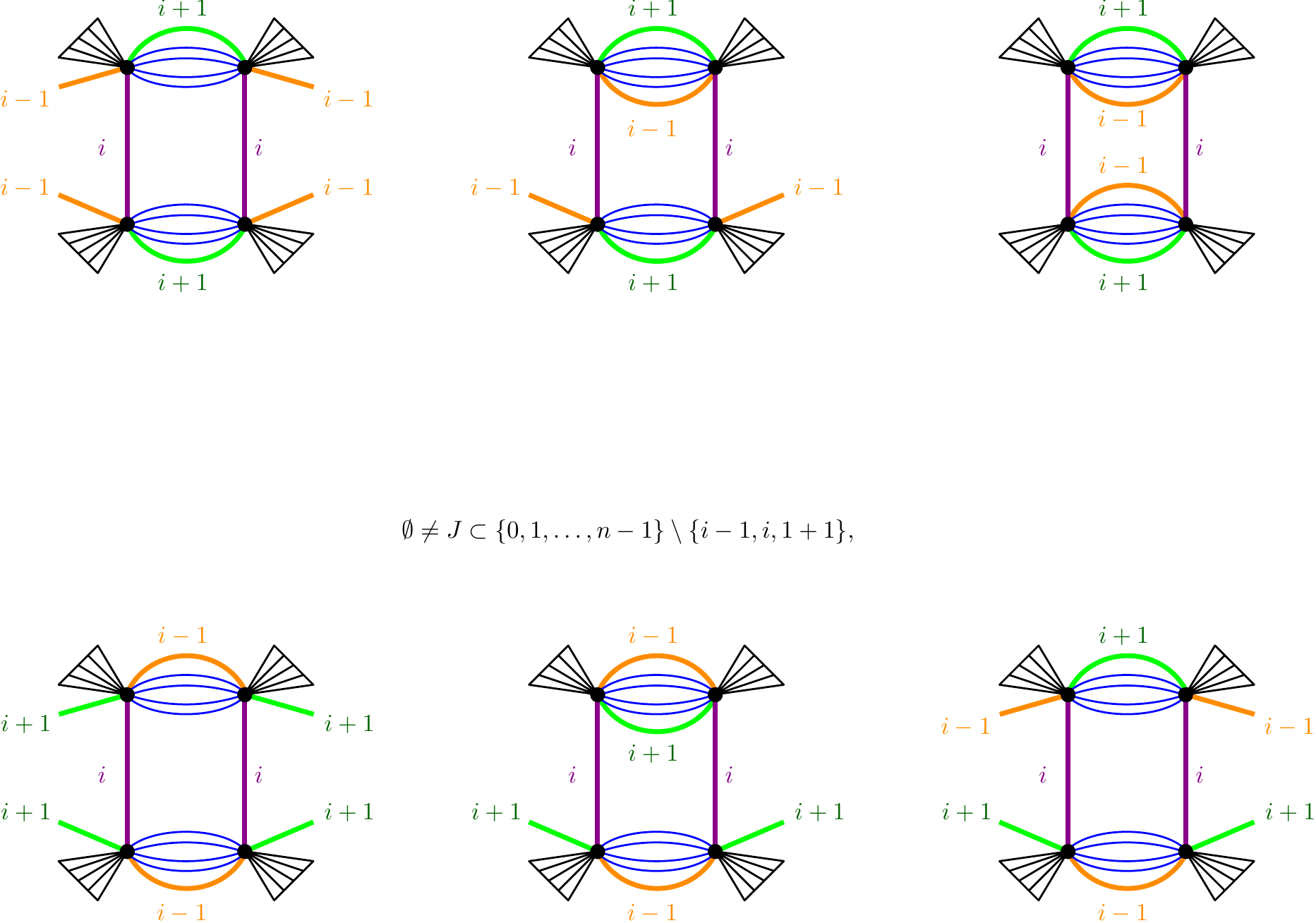}
\caption{The remaining  six families of possible symmetry type graphs of $(n-1)$-maniplexes with four orbits on its flags, and two orbits on its $i$-faces such that each contains two flag orbits.}
\label{3(n-2)_2compB}
\end{center}
\end{figure}

\newpage

We summarize our analysis of the transitivity of 4-orbit maniplexes below.

\begin{theorem}
Let $\po$ be a 4-orbit maniplex. Then, one of the following holds.
\begin{enumerate}
\item $\po$ is fully-transitive.
\item There exists $i \in \{0, \dots, n-1\}$ such that $\po$ is $j$-face-transitive for all $j\neq i$.
\item There exist $i, k \in \{0, \dots, n-1\}$, $i \neq k$, such that $\po$ is $j$-face-transitive for all $j\neq i, k$.
\item There exists $i \in \{0, \dots, n-1\}$ such that $\po$ is $j$-face-transitive for all $j\neq i, i \pm 1$.
\end{enumerate}
\end{theorem}

\subsection{On fully-transitive $n$-maniplexes for small $n$}

Every 1-maniplex is reflexible and hence fully-transitive.
Fully-transitive 2-maniplexes correspond to fully-transitive maps. It is well-known (and easy to see from the symmetry type graph) that if a map is edge-transitive, then it should have one, two or four orbits. Moreover, a fully-transitive map should be regular, a two-orbit map in class 2, $2_0$, $2_1$ or $2_2$, or a four-orbit map in class $4_{Gp}$ or $4_{Hp}$ (see, for example, \cite{medial}).

When considering fully-transitive $n$-maniplexes, $n \geq 3$, the analysis becomes considerably more complicated. 
In \cite{2-orbit} Hubard shows that there are $2^{n+1} - n -2$ classes of fully-transitive two-orbit $n$-maniplexes. By Theorem~\ref{no3-orbitfully}, there are no 3-orbit fully-transitive $n$-maniplexes. 
We note that there are 20 symmetry type graphs of 4-orbit 3-maniplexes that are fully transitive (see Figure~\ref{4orbitfully}).

\begin{figure}[htbp]
\begin{center}
\includegraphics[width=9cm]{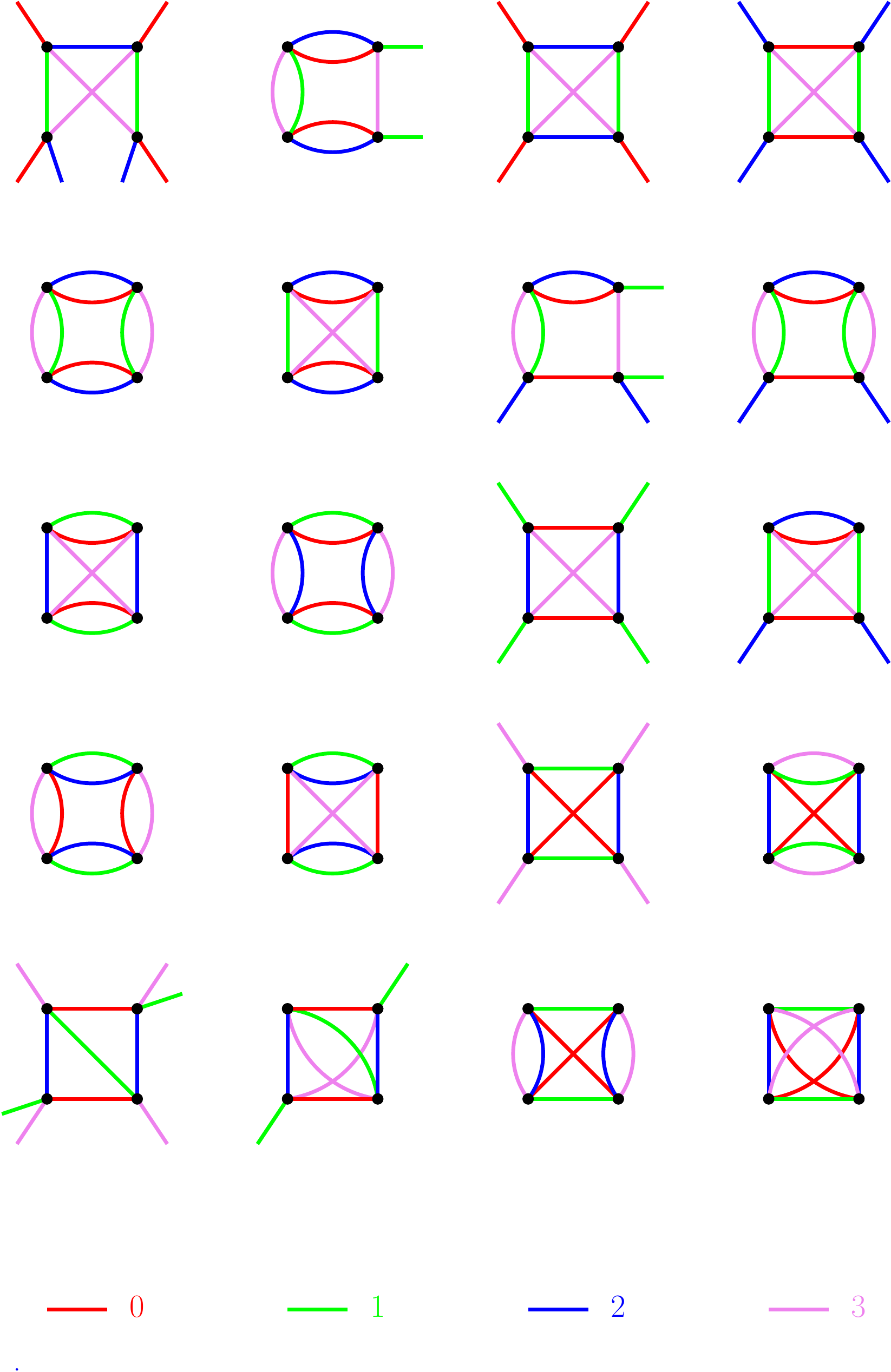}
\caption{Symmetry type graphs of 4-orbit fully-transitive 3-maniplexes}
\label{4orbitfully}
\end{center}
\end{figure}

The following theorem shall be of great use to show that a fully-transitive 3-maniplex must have an even number of flag orbits unless it is reflexible.

\begin{theorem}
Let $\po$ be a fully-transitive 3-maniplex and let $T(\po)$ be its symmetry type graph. Then either $\po$ is reflexible or $T(\po)$ has an even number of vertices.
\end{theorem}

\begin{proof}
On the contrary suppose that $T(\po)$ has an odd number of vertices, different than 1.
Whenever $|i-j|>1$, the connected components of the $(i,j)$ 2-factor of a symmetry type graph are as in Figure~\ref{4cyclequotient}. 
Hence, there is a connected component of the $(0,2)$ 2-factor of $T(\po)$ with exactly one vertex $v$ (and, hence, semi-edges of colours 0 and 2). The connectivity of $T(\po)$ implies that there is a vertex $v_1$ adjacent to $v$ in $T(\po)$. 

If $v_1$ is the only neighbour of $v$, then $T(\po)$ has the edges $(v,v_1)_1$ and $(v,v_1)_3$ as otherwise $\po$ is not fully-transitive. 
Since the connected components of the $(0,3)$ 2-factor of $T(\po)$ are as in Figure~\ref{4cyclequotient}, $v_1$ has a 0 coloured semi-edge. 
Because $T(\po)$ has more than two vertices, the edge of $v_1$ of colour 2 joins $v_1$ to another vertex, say $u$. 
But removing the edge $(v_1,u)_2$ disconnects the graph contradicting the fact that $\po$ is 2-face-transitive.

On the other hand, if $v$ has more than one neighbour it has exactly two, say $v_1$ and $u$ and $T(\po)$ has the two edges $(v,v_1)_1$ and $(v,u)_3$. This implies that the connected component of the $(1,3)$ 2-factor containing $v$ has four vertices: $v, v_1, u$ and $v_2$. (Therefore $(v_1,v_2)_3$ and $(u,v_2)_1$ are edges of $T(\po)$.) Using the $(0,3)$ 2-factor one sees that $u$ has a semi-edge of colour 0. 

Now, if $(v_1,v_2)_0$ is an edge of $T(\po)$, then the vertices $v, v_1, v_2$ and $u$ are joined to the rest of $T(\po)$ by the edges of colour 2, implying that removing them shall disconnect $T(\po)$ (there exists at least another vertex in $T(\po)$ as it has an odd number of vertices), which is again a contradiction.
On the other hand, if $v_1$ (or $v_2$) has an edge of colour 0 to a vertex $v_3$, then by Lemma~\ref{2factors4vertices} $v_2$ (or $v_1$) has a 0-edge to a vertex $v_4$. Again, if $(v_3, v_4)_1$ is an edge of $T(\po)$, since the number of vertices of the graph is odd, removing the edges of colour 2 will leave only the vertices $u,v,v_1,\dots, v_4$ in one component, which is a contradiction. Proceeding now by induction on the number of vertices one can conclude that $T(\po)$ cannot have an odd number of vertices 
\end{proof}

\section{Generators of the automorphism group of a $k$-orbit maniplex}
\label{Gen-autG}

It is well-known among polytopists that the automorphism group of a regular $n$-polytope can be generated by $n$ involutions. In fact, given a base flag $\Phi \in \fl(\poo)$, the distinguished generators of $\Aut(\po)$ with respect to $\Phi$ are involutions $\rho_0, \rho_1, \dots, \rho_{n-1}$ such that $\Phi \rho_i = \Phi^i$.

Generators for the automorphism group of a two-orbit $n$-polytope can also be given in terms of a base flag (see \cite{2-orbit}). In this section we give a set of distinguished generators (with respect to some base flag) for the automorphism group of a $k$-orbit $(n-1)$-maniplex in terms of the symmetry type graph $T(\po)$, provided that $T(\po)$ has a hamiltonian path.

Given two walks $w_{1}$ and $w_{2}$ along the edges and semi-edges of $T(\mathcal{M})$ such that the
final vertex of $w_{1}$ is the starting vertex of $w_{2}$, we define
the sequence $w_{1}w_{2}$ as the walk that traces all the edges of
$w_{1}$ and then all the edges of $w_{2}$ in the same order; the
inverse of $w_{1}$, denoted by $w_{1}^{-1}$, is the walk which has
the final vertex of $w_{1}$ as its starting vertex, and traces all
the edges of $w_{1}$ in reversed order. Since each of the elements
of $\mathrm{Mon}(\mathcal{M})$ associated to the edges of $T(\mathcal{M})$
is its own inverse, we shall forbid walks that trace the same edge
two times consecutively (or just remove the edge from such
walk, shortening its length by two). 
Given a set of walks in
$T(\mathcal{M})$, we say that a subset $\W'\subseteq \W$ is {\em a
generating set of $\W$} if each $w\in \W$ can be expressed as a sequence
of elements of $\W'$ and their inverses.
Now, let $\W$ be the set of closed walks along the edges and semi-edges
of $T(\mathcal{M})$ starting at a distinguished vertex $v_0$.
Recall that the walks along the edges and semi-edges of $T(\mathcal{M})$
correspond to permutations of the flags of $\mathcal{M}$; moreover,
each closed walk of $\W$ corresponds to an automorphism of $\mathcal{M}$.
Thus, by finding a generating set of $\W$, we will find a set of automorphisms
of $\mathcal{M}$ that generates $\Aut(\mathcal{M})$. (However, the converse is not true, as an automorphism of $\po$ may be described in more than one way as a closed walk of $T(\po)$.) Given $T(\mathcal{M})$, we may easily find such generating set. The construction
goes as follows:

Let $\mathcal{M}$ be a $k$-orbit maniplex of rank $n-1$ such that
 $\mathcal{C}=(v_{0},v_{1},v_{2},...,v_{q})$ is a walk of minimal length that visits all the vertices of $T(\ma)$.
The sets of vertices and edges (and semi-edges) of $T(\mathcal{M})$
will be denoted by $V$ and $E$, respectively. The set of edges visited
by $\mathcal{C}$ will be denoted by $E_{\mathcal{C}}$.  
In this section, the edges
joining two vertices $v_{i}$ and $v_{j}$ will be denoted by $(v_{i},v_{j})_{1}$,
$(v_{i},v_{j})_{2}$, $(v_{i},v_{j})_{3}$,...,$(v_{i},v_{j})_{h}$;
if $j=i+1$ then $(v_{i},v_{j})_{1}\in E_{\mathcal{C}}$. 
(Note that in order to not start carrying many subindices, we  modify the notation of the edges of $T(\ma)$ that we had used throughout the paper. If one wants to be consistent with the notation of the edges used in the previous sections, one would have to say that the edges between $v_i$ and $v_j$ are $(v_i,v_j)_{a_1}$, $(v_i,v_j)_{a_2}, \dots (v_i,v_j)_{a_h}$). 
Similarly,
we  denote all semi-edges incident to a vertex $v_{i}$ by $(v_{i},v_{i})_{1}$,
$(v_{i},v_{i})_{2}$, $(v_{i},v_{i})_{3}$,...,$(v_{i},v_{i})_{l}$.
For the sake of simplicity, $(v_{i},v_{j})_{1}$ will be just called
$(v_{i},v_{j})$. Let $\W$ be the set of all closed walks in $T(\mathcal{M})$
with $v_{0}$ as its starting vertex. We shall now construct $G(\W)\subseteq \W$,
a generating set of $\W$.

For each edge $(v_{i},v_{j})_{m}\in E\setminus E_{\mathcal{C}}$ we
shall define the walk 
$w_{i,j,m}=((v_{0},v_{1}),(v_{1},v_{2}),...,(v_{i-1},v_{i}),(v_{i},v_{j})_{m},(v_{j},v_{j-1}),(v_{j-1},v_{j-2}),...,(v_{1},v_{0})).$
 That is, we walk from $v_0$ to $v_i$ in $E_{\mathcal{C}}$, and then we take the edge $(v_{i},v_{j})_{m}$, and then we walk back from
$v_j$ to $v_0$ in $E_{\mathcal{C}}$.
Let $\W_{e}\subseteq \W$ be the set of all such walks.

For each semi-edge $(v_{i},v_{i})_{l}\in E\setminus E_{\mathcal{C}}$
we shall define the walk $w_{i,i,l}=((v_{0},v_{1}),(v_{1},v_{2}),...,(v_{i-1},v_{i}),(v_{i},v_{i})_{l},(v_{i},v_{i-1}),(v_{i-1},v_{i-2}),...,(v_{1},v_{0}))$.
That is, we walk from $v_0$ to $v_i$ in $E_{\mathcal{C}}$, and then we take the semi-edge $(v_{i},v_{i})_{l}$, and then we walk back from
$v_i$ to $v_0$ in $E_{\mathcal{C}}$.
Let $\W_{s}\subseteq \W$ be the set of all such walks.

We define $G(\W)=\W_{e}\cup \W_{s}$.

\begin{lemma}
\label{generatingwalks}
With the notation from above, $G(\W)$ is a generating set for $\W$.
\end{lemma}

\begin{proof}
We shall prove that any $w\in \W$ can be expressed
as a sequence of elements of $G(\W)$ and their inverses. Let $w\in \W$
be a closed walk among the edges and semi-edges of $T(\mathcal{M})$ starting at $v_0$.
From now on, semi-edges will be referred to simply as ``edges''.

We shall proceed by induction over $n$, the number of edges in $E\setminus E_{\mathcal{C}}$
visited by $w$. If $w$ visits only one edge in $E\setminus E_{\mathcal{C}}$,
then $w\in G(\W)$ or $w^{-1}\in G(\W)$. Let us suppose that,
if a closed walk among the edges of $T(\mathcal{M})$
visits $m$ different edges in $E\setminus E_{\mathcal{C}}$, with
$m<n$, then it can be expressed as a sequence of elements of $G(\W)$
and their inverses.

Let $w\in \W$ be a walk that visits exactly $n$ edges in $E\setminus E_{\mathcal{C}}$.
Let $(v_{a},v_{b})_{l}\in E\setminus E_{\mathcal{C}}$ be the last 
edge of $E\setminus E_{\mathcal{C}}$
visited by $w$. 
Without loss of generality we may assume that the vertex
$v_{b}$ was visited after $v_{a}$, so let $(v_c, v_a)_m$ be the edge that $w$ visits just before $(v_a,v_b)_l$ (note that $(v_c, v_a)_m$ may or may not be in $E_{\mathcal{C}}$). 
Let $w_{1}\in \W$ be the closed walk that
traces the same edges (in the same order) as $w$ until reaching
$(v_{c},v_{a})_{m}$ and then traces the edges $(v_{a},v_{a-1})$,
$(v_{a-1},v_{a-2})$, ...,$(v_{1},v_{0})$, and let $w_{2}\in \W$
be the closed walk  
that traces the edges $(v_{0},v_{1}),(v_{1},v_{2}),...,(v_{a-1},v_{a})$ and then traces $(v_a,v_b)_l$ and continues the way $w$ does to return to $v_0$.
It is clear that $w_{1}$ visits exactly $n-1$ edges in $E\setminus E_{\mathcal{C}}$
and that $w_{2}$ visits only one. 
By inductive hypothesis both $w_{1}$
and $w_{2}$ can be expressed as a sequence of elements of $G(\W)$,
and therefore so does $w$ since $w=w_{1}w_{2}$.
\end{proof}

Let $\Phi$ be a base flag of $\po$ that projects to the initial vertex of a walk that contains all vertices of $T(\ma)$ of a symmetry type graph.
 Following the notation of \cite{d-auto}, given $w \in \Mon(\po)$ such that $\Phi^w$ is in the same orbit as $\Phi$ (that is, $ w \in \mathrm{Norm(Stab (\Phi))}$), we denote by $\alpha_w$ the automorphism taking $\Phi$ to $\Phi^w$. Moreover, if $w=r_{i_1}r_{i_2}\dots r_{i_k}$ for some $i_1, \dots i_k \in \{0, \dots, n-1\}$, then we may also denote $\alpha_w$ by $\alpha_{i_1,i_2, \dots i_k}$.
 
The following theorem gives distinguished generators (with respect to some base flag) of the automorphism group of a maniplex $\ma$  in terms of a distinguished walk of $T(\po)$, that travels through all the vertices of $T(\po)$. Its proof is a consequence of the previous lemma.
 
 \begin{theorem}
 \label{auto}
 Let $\po$ be a $k$-orbit $n$-maniplex and let $T(\po)$ its symmetry type graph. 
 Suppose that  $v_1, e_1, v_2, e_2 \dots, e_{q-1}, v_q$ is a distinguished walk that visits every vertex of $T(\ma)$, with the edge $e_i$ having colour $a_i$, for each $i= 1, \dots q-1$.
 Let $S_i \subset \{0, \dots, n-1\}$ be such that $v_i$ has a semi-edge of colour $s$ if and only if $s \in S_i$. 
 Let $B_{i,j}\subset \{0, \dots, n-1\}$ be the set of colours of the edges between the vertices $v_i$ and $v_j$ (with $i<j$) that are not in the distinguished walk 
 and
  let $\Phi \in \fl(\ma)$ be a base flag of $\po$ such that $\Phi$ projects to $v_1$ in $T(\po)$.
 Then, the automorphism group of $\ma$ is generated by the union of the sets
 $$\{\alpha_{a_1, a_2, \dots, a_i, s, a_i, a_{i-1}, \dots, a_1} \mid i=1, \dots, k-1, s \in S_i \},$$
 and
 $$\{ \alpha_{a_1, a_2, \dots, a_i, b, a_j, a_{j-1}, \dots, a_1} \mid i,j \in \{1, \dots, k-1\}, i<j, b \in B_{i,j} \}.$$
 \end{theorem}
 
We note that, in general, a set of generators of $\Aut(\po)$ obtained from Theorem~\ref{auto} can be reduced since
there might be more than one element of $G(\W)$ representing the same automorphism.
For example, the closed walk $w$ through an edge of colour $2$, then a $0$-semi-edge and finally a 2-edge corresponds to the element $r_2r_0r_2 = r_0$ of $\Mon(\ma)$. Hence, the group generator induced by the walk $w$ is the same as that induced by the closed walk consisting only of the semi-edge of colour $0$.

The following two corollaries give a set of generators for 2- and 3-orbit polytopes, respectively, in a given class. The notation follows that of Theorem~\ref{auto}, where if the indices of some $\alpha$ do not fit into the parameters of the set, we understand that such automorphism is the identity.

\begin{corollary}
{\bf \cite{tesisisa}}
Let $\po$ be a 2-orbit $(n-1)$-maniplex in class $2_I$, for some $I \subset \{0, \dots, n-1\}$ and let $j_0 \notin I$. Then
$$ \big\{ \alpha_i, \alpha_{j_0, i, j_0},  \alpha_{k,j_0} \mid  i \in I,  \ k \notin I \big\}$$
is a generating set for $\Aut(\po)$.

\end{corollary}

\begin{corollary}
Let $\po$ be a 3-orbit $(n-1)$-maniplex.
\begin{enumerate}
\item If $\po$ is in class $3^{i}$, for some $i \in \{1, \dots, n-2\}$, then 
$$ \big\{ \alpha_j, \alpha_{i,i-1,i+1,i},  \alpha_{i,i+1,i+2,i+1,i},  \alpha_{i,i+1,i,i+1,i} \mid  j \in \{0, \dots, n-1\} \setminus \{i\} \big\} $$
is a generating set for $\Aut(\po)$.
\item If $\po$ is in class $3^{i,i+1}$, for some $i \in  \{0, \dots, n-2\}$, then 
$$\big\{ \alpha_j, \alpha_{i,i-1,i}, \alpha_{i,i+1,i+2,i+1,i},  \alpha_{i,i+1,i,i+1,i} \mid  j \in \{0, \dots, n-1\} \setminus \{i\} \big\}$$
is a generating set for $\Aut(\po)$.
\end{enumerate}
\end{corollary}

\section{Oriented and orientable maniplexes}
\label{sec:orient}

A maniplex $\po$ is said to be {\em orientable} if its flag graph $\gr$ is a bipartite graph. 
Since a subgraph of a bipartite graph is also bipartite, all the sections of an orientable maniplex are orientable maniplexes themselves.
An {\em orientation} of an orientable maniplex is a colouring of the parts of $\gr$, with exactly two colours, say black and white. An {\em oriented maniplex} is an orientable maniplex with a given orientation.
Note that any oriented maniplex $\po$ has an enantiomorphic maniplex (or mirror image) $\po^{en}$. One can think of the enantiomorphic form of an oriented maniplex simply as the orientable maniplex with the opposite orientation. 

If the connection groups $\Mon(\po)$ of $\po$ is generated by $r_0, r_1, \dots, r_{n-1}$,
for each $i \in \{0, \dots, n-2\}$ let us define the element $t_i:= r_{n-1}r_i \in \Mon(\po)$. Then, $t_i^2=1$, for $i=0, \dots n-3$. The subgroup $\Mon^+(\po)$ of $\Mon(\po)$ generated by $t_0, \dots t_{n-2}$ is called {\em even connection group of $\po$}. Note that $\Mon^+(\po)$ has index at most two in $\Mon(\po)$. 
In fact $(\Mon^+(\po))^{r_{n-1}} = \Mon^+(\po^{en})$.
It should be clear then that any maniplex and its enantiomorphic form are in fact isomorphic as maniplexes.

An {\em oriented flag di-graph} $\gr^+$ of an oriented maniplex $\po$ is constructed in the following way. 
The vertex set of $\gr^+$ consists of one of the parts of the bipartition of $\gr$. That is, the black (or white) vertices of the flag graph of $\po$.  The darts of $\gr^+$ will be the 2-arcs of $\gr$ of colours $n-1,i$, for each $i \in\{0,\dots, n-2\}$. We then identify two darts to obtain an edge if they have the same vertices, but go in opposite directions. 
Note that for $i = 0, \dots, n-3$ and each flag $\Phi$ of $\po$, the 2-arc starting at $\Phi$ and with edges coloured $n-1$ and $i$ has the same end vertex than the 2-arc starting at $\Phi$ and with edges coloured $i$ and $n-1$.
Hence,  all the darts corresponding to 2-arcs of colours ${n-1}$ and $i$, with $i=0, \dots n-3$ will have both directions in $\gr^+$ giving us, at each vertex, $n-2$ different edges. 
On the other hand, the 2-arcs on edges of two colours ${n-1}, {n-2}$ will in general be directed darts of $\gr^+$. An example of an oriented flag di-graph is shown in Firgure~\ref{orientedgraphflag}. 
We note that the oriented flag di-graph of $\po^{en}$ can be obtained from $\gr^+$ by reversing the directions of the $n-2, n-1$ darts.

\begin{figure}[htbp]
\begin{center}
\includegraphics[width=10cm]{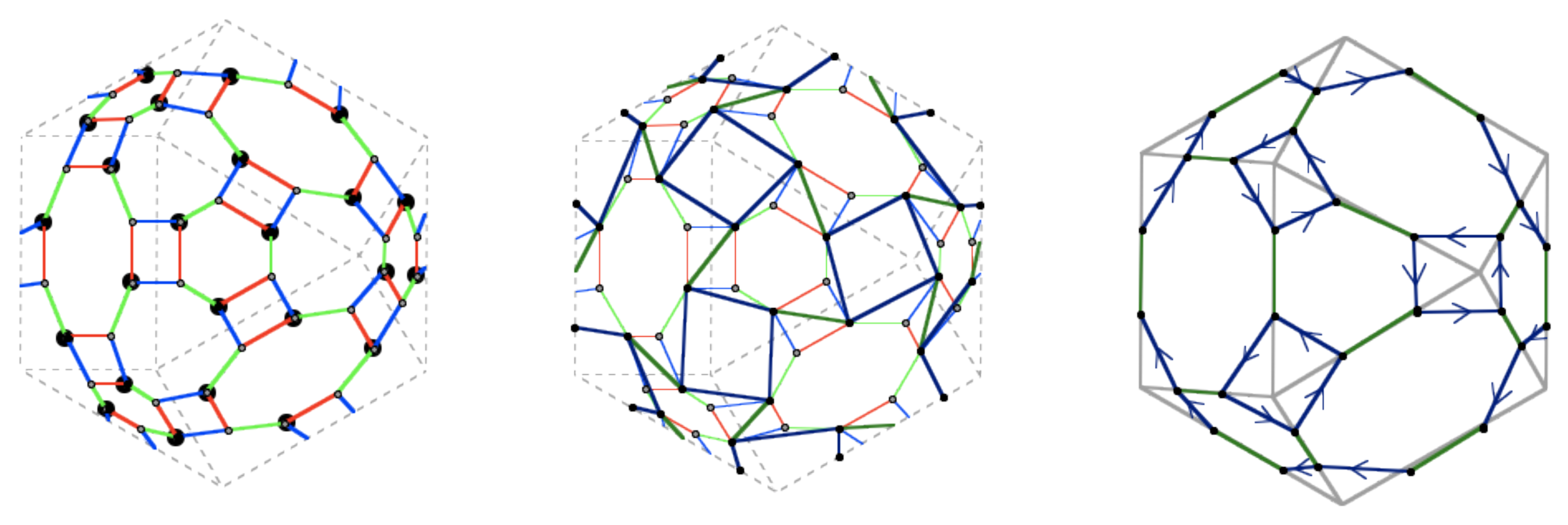}
\caption{The oriented flag di-graph of an oriented cuboctahedron from its flag graph.}
\label{orientedgraphflag}
\end{center}
\end{figure}

Note that the 2-arcs of colours $r_{n-1}, r_i$ correspond to the generators $t_i$ of $\Mon^+(\po)$. In fact, as $\Mon^+(\po)$ consists precisely of the even words of $\Mon(\po)$, a maniplex is orientable if and only if the index of $\Mon^+(\po)$ in $\Mon(\po)$ is exactly two. 
We can then colour the edges and darts of $\gr^+$ with the elements $t_i$.
The fact that $t_i^2=1$ for every $i=0, \dots, n-3$ indeed implies that the edges of $\gr^+$ are labelled by these first $n-2$ elements, while the darts are labelled by $t_{n-2}$.

We can see now that for each $i \in \{0, \dots, n-2\}$, the $i$-faces of $\po$ are in correspondence with the connected components of the subgraph of $\gr^+$ with edges of colours $\{0, \dots, n-2\}\setminus \{i\}$. To identify the facets of $\po$ as subgraphs of $\gr^+$, we first consider some oriented paths on the edges of $\gr^+$. We shall say that an oriented path on the edges of $\gr^+$ is {\em facet-admissible}
if no two darts of colour $t_{n-2}$ are consecutive on the path. Then, two vertices of $\gr^+$ are in the same facet of $\po$ if there exists a facet admissible oriented path from one of the vertices to the other.

For the remainder of this section, by a maniplex we shall mean an oriented maniplex, with one part of the flags coloured with black and the other one in white.

An {\em orientation preserving automorphism} of an (oriented) maniplex $\po$  is an automorphism of $\po$ that sends black flags to black flags and white flags to white flags. An {\em orientation reversing automorphism} is an automorphism that interchanges black and white flags.
 A {\em reflection} is an orientation reversing involutory automorphism. The group of orientation preserving automorphisms of $\po$ shall be denoted by $\Aut^+(\po)$. 
 
 The orientation preserving automorphism $\Aut^+(\po)$ of a maniplex $\po$ is a subgroup of index at most two in $\Aut(\po)$. In fact, the index is exactly two if and only if $\Aut(\po)$ contains an orientation reversing automorphism. Note that in this case, there exists an orientation reversing automorphism that sends $\po$ to its enantiomorphic form $\po^{en}$.

Pisanski~\cite{tomo} defines a maniplex to be {\em chiral-a-la-Conway} if $\Aut^+(\po) = \Aut(\po)$. 
If a maniplex $\po$ is chiral-a-la-Conway, then its enantiomorphic maniplex $\po^{en}$ is isomorphic to $\po$, but there is no automorphism of the maniplex sending one to the other. 
 It follows from the definition that $\po$ is chiral-a-la-Conway if and only if the automorphisms of $\po$ preserve the bipartition of $\gr$ and therefore we have the following proposition.
 
 \begin{proposition}
 \label{oddcycles}
Let $\po$ be an oriented maniplex and let $T(\po)$ its symmetry type graph. Then, $\po$ is chiral-a-la-Conway if and only if $T(\po)$ has no odd cycles.
\end{proposition}
 
 Similarly as before, the orientation preserving automorphisms of a maniplex $\po$ correspond to colour preserving automorphism of the bipartite graph $\gr$ that preserves the two parts. But these correspond to colour preserving automorphisms of the di-graph $\gr^+$, implying that $\Aut^+(\po) \cong \Aut_p(\gr^+)$.
Note that the action of $\Aut^+(\po)$ on the set $\mathcal{B(\po)}$ of all the black flags of $\po$ is semiregular, and 
hence, the action on $\Aut_p(\gr^+)$ is semiregular on the vertices of $\gr^+$.

An oriented maniplex $\po$ is said to be {\em rotary (or orientably regular)} if the action of $\Aut^+(\po)$ is regular on $\bl(\po)$. 
Equivalently, $\po$ is rotary if the action of $\Aut_p(\gr^+)$ is regular on its vertices.  
We say that $\po$ is {\em orientably $k$-orbit} if the action of $\Aut_p(\gr^+)$ has exactly $k$ orbits on the vertices of $\gr^+$. The following lemma is straightforward.

\begin{lemma}
\label{orientablekorbit}
Let $\po$ be a chiral-a-la-Conway maniplex. Then $T(\po)$ has no semi-edges and if $\po$ is an orientably $k$-orbit maniplex, then $\po$ is a $2k$-orbit maniplex.
\end{lemma}

\subsection{Oriented symmetry type di-graphs of oriented maniplexes}

We now consider the semiregular action of $\Aut^+(\po)$ on the vertices of $\gr^+$, and let ${\mathcal B}= {\mathcal O}rb^+$ be the partition of the vertex set of $\gr^+$ into the orbits with respect to the action of $\Aut^+(\po)$.  (As before, since the action is semiregular, all orbits are of the same size.) The {\em oriented symmetry type di-graph} $T^+(\po)$ of $\po$ is the quotient colour di-graph with respect to ${\mathcal O}rb^+$. Similarly as before, if $\po$ is rotary, then the oriented symmetry type di-graph of $\po$ consists of one vertex with one loop and $n-2$ semi-edges. 
Note that for oriented symmetry type di-graphs we shall not identify two darts with the same vertices, but different directions.

If we now turn our attention to oriented symmetry type di-graphs with two vertices, one can see that for each $I\subset \{0, \dots, n-2\}$, there is an oriented symmetry type di-graphs with two vertices having semi-edges (or loops) of colours $i$ at each vertex for every $i \in I$, and having edges (or both darts) of colour $j$, for each $j \notin I$. An oriented maniplex with such oriented symmetry type di-graph shall be say to be in class $2_I^+$. Hence, there are $2^{n-2}-1$ classes of oriented 2-orbit $(n-1)$-maniplexes.

Note that if $\ma$ is a $k$-orbit maniplex, then $T^+(\po)$ has either $k$ or $\frac{k}{2}$ vertices.
The next result follows from Proposition~\ref{oddcycles} and Lemma~\ref{orientablekorbit}.

\begin{theorem}
Let $\po$ be an oriented maniplex. Then, $T(\po)$ and $T^+(\po)$ have the same number of vertices if and only if $T(\po)$ has a semi-edge or an odd cycle.
\end{theorem}

It is not difficult to see that if we are to consider for a moment an oriented symmetry type di-graph $T^+$ with an (undirected) hamiltonian path, then the construction of Section~\ref{Gen-autG} gives us a way to construct a generating set of the closed walks based at the starting vertex of the path (and Lemma~\ref{generatingwalks} implies that the set actually generates.) Hence, one can find generators for the group of orientation preserving automorphisms of an oriented maniplex, provided that it has an (undirected) hamiltonian path. In particular we have the following theorem.

\begin{theorem}
Let $\po$ be an oriented 2-orbit $(n-1)$-maniplex in class $2_I^+$, for some $I\subset \{0, \dots, n-2\}$. Then
\begin{enumerate}
\item If $n-2 \in I$, let  $j_0 \notin I$, then
$$\big\{ 
\alpha_{i,n-1}, \alpha_{j_0, n-1, i, j_0}, \alpha_{k, n-1, j_0, n-1} \mid  i \in I  \ k \notin I
\big\}$$
is a generating set for $Aut^+(\po)$.
\item If $n-2 \notin I$ but there exists $j_0 \notin I$, $j_0 \neq n-2$, then
$$\big\{ 
\alpha_{i,n-1}, \alpha_{j_0, n-1, i, j_0}, \alpha_{k, n-1, j_0, n-1} , \alpha_{n-1, n-2, j_0, n-1}\mid  i \in I  \ k \notin I
\big\}$$
is a generating set for $Aut^+(\po)$.
\item If $I =\{0, \dots, n-3\}$, then
$$\big\{ 
\alpha_{i,n-1}, \alpha_{n-2, n-1, i, n-1, n-2}, \alpha_{n-1, n-2, n-1, n-2} \mid  i \in I
\big\}$$
is a generating set for $Aut^+(\po)$.
\end{enumerate}
\end{theorem}

Given an oriented maniplex $\ma$ and its symmetry type graph $T(\ma)$, we shall say that $T^+(\ma)$ is the associated oriented symmetry type di-graph of $T(\po)$. Hence, given a symmetry type graph $T$ one can find its associated oriented symmetry type di-graph  $T^+$ by erasing all edges of $T$ and replacing them by the $n-1,i$ paths of $T$. Note that this replacement of the edges may disconnect the new graph. If that is the case, we take $T^+$ to be one of the connected components.

\subsection{Oriented symmetry type graphs with three vertices}

In a similar way as one can classify maniplexes with small number of flag orbits (under the action of the automorphism group of the maniplex) in terms of their symmetry type graph, one can classify oriented maniplexes with small number of flags (under the action of the orientation preserving automorphism group of the maniplex) in terms of their oriented symmetry type di-graph.

Let $\mathcal{M}$ be a $6$-orbit Chiral-a-la-Conway $(n-1)$-maniplex, with $n\geq4$. 
Let $T(\mathcal{M})$ be its symmetry type graph and $T^{+}(\mathcal{M})$ be its oriented symmetry type di-graph. 
Recall that $T(\mathcal{M})$ is a graph with 6 vertices and no semi-edges
or odd cycles, and that $T^{+}(\mathcal{M})$ is a di-graph with 3
vertices. 
Let $V=\{v_{1},v_{2},..,v_{6}\}$ be the vertex set of $T(\mathcal{M})$.
We may label the vertices of $T(\po)$ in such a way that
the edges $(v_{1},v_{2})$, $(v_{3},v_{4})$, $(v_{5},v_{6})$ are
coloured with the colour $(n-1)$, and that no two vertices of the
set $\{v_{1},v_{3},v_{5}\}$ are adjacent. Let $\W=\{w_{1},w_{3},w_{5}\}$
be the vertex set of $T^{+}(\po)$. Each $w_{i}\in \W$ corresponds
to the vertex $v_{i}\in V$, $i\in\{1,3,5\}$. 
In what follows, in the same way as in Section~\ref{sec:stg}, $(v_{i},v_{j})_{k}$
 denotes the $k$-coloured edge joining the vertices $v_{i}$ and
$v_{j}$, $v_{i},v_{j}\in V$, $k\in\{0,1,...,n-1)$; and $(w_{i},w_{j})_{k}$
 denotes the $(k,n-1)$-coloured edge joining the vertices $w_{i}$
and $w_{j}$, $w_{i},w_{j}\in \W$ and $k\in\{0,1,...,n-3)$.

Since there are no semi-edges in $T(\po)$, for each colour
$i\in\{0,...,n-3\}$ there is one edge (and one semi-edge) of colour
$(i,n-1)$ in $T^{+}(\po)$ if and only if the 2-factor of
$T(\po)$ of colours $i$ and $(n-1)$ consists of one 4-cycle
and one 2-cycle of alternating colours. Likewise, there are three
semi-edges of colour $(i,n-1)$ in $T^{+}(\po)$ if and only
if the 2-factor of $T(\po)$ of colours $i$ and $(n-1)$
consist of three 2-cycles. 
It is straightforward to see that there
are two consecutive edges of colour $(i,n-1)$ and $(j,n-1),$ $i\neq j$,
in $T^{+}(\po)$ if and only if the 2-factor of colours $i$
and $j$ consists of a single 6-cycle. It follows that if there are
two consecutive edges of colour $(i,n-1)$ and $(j,n-1)$ in $T^{+}(\po)$,
then $\left|i-j\right|<2$. 

Notice that the possible 2-factors of colour $(n-1)$ and $(n-2)$
in $T(\po)$ are either a single 6-cycle of alternating colours,
a 4-cycle along with a 2-cycle, or three separate 2-cycles. Hence,
the darts in $T^{+}(\po)$ are arranged in either a 3-cycle,
a 2-cycle along with a loop, or three separate loops. We proceed case
by case.

Consider the case when there are three loops in $T^{+}(\po)$.
Since oriented symmetry type di-graphs are connected, then without
loss of generality $(w_{1},w_{3})_{i}$ and $(w_{3},w_{5})_{i+1}$
must be edges of $T^{+}(\po)$. We may suppose that $(w_{1},w_{3})_{i}$
is the only edge joining $w_{1}$ and $w_{3}$. If there is a third
edge in $T^{+}(\po)$, then it is necessarily $(w_{3},w_{5})_{i-1}$.
Note that, since the edges coloured by $(n-1)$ and $(n-2)$ do not
lie on a 6-cycle in $T(\po)$, there are no restrictions on
the semi-edges of $T^{+}(\po)$. Thus, there is one oriented
symmetry type di-graph for each pair of colours $i$ and $i+1$, with
$i\in\{0,...,n-3\}$ and one for each triple $i-1$, $i$ and $i+1$,
$i\in\{1,...,n-3\}$. Therefore, there are $2n-7$ oriented symmetry
type di-graphs with 3 loops.

Consider the case when $T^{+}(\po)$ has only one loop. We
may suppose that the loop is in $w_{5}$ and the vertices $w_{1}$
or $w_{3}$ are joined by darts. This implies that $(v_{1},v_{4})_{(n-2)}$,
$(v_{2},v_{3})_{(n-2)}$ and $(v_{5},v_{6})_{(n-2)}$ are edges of
$T(\po)$. As $T^{+}(\po)$ is connected, there must
be an edge joining $w_{3}$ and $w_{5}$ of colour $(i,n-1).$ 
Necessarily
$i=n-3$, since the edges $(v_{1},v_{2})_{i}$, $(v_{2},v_{3})_{(n-2)}$,
$(v_{3},v_{6})_{i}$, $(v_{6},v_{5})_{n-2}$, $(v_{5},v_{4})_{i}$,
$(v_{4},v_{1})_{n-2}$ form a 6-cycle in $T(\po)$. Notice
that there are no restrictions on the semi-edges of $T^{+}(\po)$.
Hence, there are exactly two oriented symmetry type di-graph with
a single loop: one with a single edge of colour $(n-3,n-1)$ between $w_3$ and $w_5$, and
one with two edges of colours $(n-3,n-1)$ and $(n-4,n-1)$ between them.

Consider the case when the darts in $T^{+}(\po)$ are arranged
in a 3-cycle. It is clear that the 2-factor of $T(\po)$ of
colours $(n-2)$ and $(n-1)$ is a single 6-cycle. Therefore, if $i\in\{0,...,n-4\}$,
the 2-factor of $T(\po)$ of colours $i$ and $(n-1)$ cannot
consist of three 2-cycles, as this implies the existence of a 6-cycle
of alternating colours $i$ and $(n-2)$ such that $\left|i-(n-2)\right|\geq2$.
That is, $T^{+}(\po)$ has one edge (and one semi-edge) of
colour $(i,n-1)$ for each $i\in\{0,...,n-4\}$ and either one edge
and a semi-edge, or three semi-edges for colour $(n-3,n-1).$ Note
that if $n\geq7$, the set $\{0,...,n-4\}$ has more than three elements
and thus all edges of colour $(i,n-1)$ in $T^{+}(\po)$, $i\in\{0,...,n-4\}$,
must be joining the same pair of vertices. Otherwise, there would
be at least two consecutive edges of colours $(i,n-1)$ and $(j,n-1)$,
with $\left|i-j\right|\geq2$. Figure~\ref{3orient} below shows the only four
possible oriented symmetry type di-graphs with a 3-cycle of darts
and at least two consecutive edges. Two correpond to 4-maniplexes,
one to 3-maniplexes and one to 5-maniplexes. These will be treated as
special cases.

\begin{figure}[htbp]
\begin{center}
\includegraphics[width=10cm]{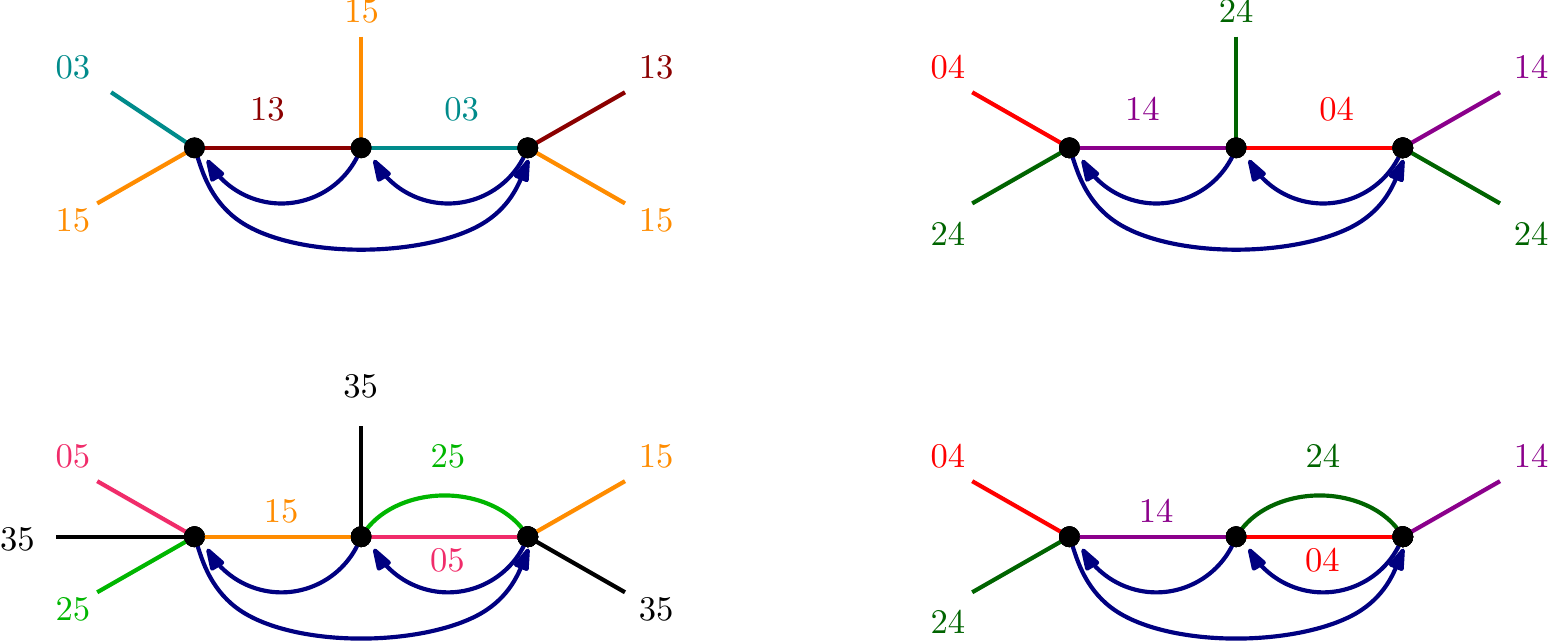}
\caption{Oriented symmetry type di-graphs with 3 vertices and one directed 3-cylce, of 3-, 4- and 5-maniplexes}
\label{3orient}
\end{center}
\end{figure}

We may suppose that $T^{+}(\po)$ has no consecutive edges.
It follows that here are exactly two oriented symmetry type di-graph
with a 3-cycle of darts: one with an edge joining the same pair of
vertices for each colour $i\in\{0,...,n-3\},$ and one with three
semi-edges of colour $(n-3,n-1)$ and an edge joining the same pair
of vertices for each colour $i\in\{0,...,n-4\}$.

Considering all the cases above, there are $(n-3)+(n-4)+2+2=2n-3$
oriented symmetry type graphs with three vertices for oriented maniplexes of
rank $n\geq6$; $2n-2=6$ for oriented maniplexes of rank 3; $2n-1=9$ for oriented maniplexes
of rank 4; and $2n-2=10$ for oriented maniplexes of rank 5.

\section*{Acknowledgments}

This work was done with the support of ''Programa de Apoyo a Proyectos de Investigaci\'on e Innovaci\'on Tecnol\'ogica (PAPIIT) de la UNAM,  IB101412 {\em Grupos y gr\'aficas asociados a politopos abstractos}". The second author was partially supported by Slovenian Research Agency (ARRS) and the third author was  partially supported by CONACyT under project 166951 and by the program ''Para las mujeres en la ciencia L'Oreal-UNESCO-AMC 2012"


\begin{thebibliography}{99}


\bibitem{CompSymTypeGraph}
 Brinkmann, G., Van Cleemput, N., Pisanski, T.
 {\em Generation of various classes of trivalent graphs}.
 Theoretical Computer Science. (2012). 
In press. 
 
 \bibitem{trunc} 
Del R\'io-Francos, M. 
{\em Truncation symmetry type graphs}. 
In preparation.

\bibitem{medial} 
Del R\'io-Francos, M., Hubard I., Orbanic A., Pisanski T. 
{\em Medial symmetry type graphs.} 
arXiv:1301.7637 [math.CO]


\bibitem{2-orbit} 
Hubard I. 
{\em Two-orbit polyhedra from groups}
European Journal of Combinatorics, \textbf{31 (39} (2010), 943--960

\bibitem{tesisisa}
Hubard I.
{\em From geometry to groups and back: The study of highly symmetric polytopes}, 
ProQuest LLC, Ann Arbor, MI, 2007, Thesis (Ph.D.)�York University (Canada). MR 2712832

\bibitem{d-auto} 
Hubard I., Orbanic A., Weiss A.I, 
{\em Monodromy groups and self-invariance}, 
Canadian Journal of Mathematics {\bf 61} (2009), 1300--1324.

\bibitem{Archim} 
Koci\v{c}, J. 
{\em Symmetry-type graphs of Platonic and Archimedean solids},
Mathematical Communications {\bf 16} (2011), 491�507.

 
 \bibitem{arp} 
McMullen P., Schulte E.
 {\em Abstract Regular Polytopes},
 Cambridge University Press, (2002).

\bibitem{k-orbitM}
 Orbani\'c A., Pellicer D., Weiss A. I.
 {\em Map operation and $k$-orbit maps}.
 Journal of  Combinatorial Theory, Series A \textbf{117 (4)}  (2009) 411--429.
 
 \bibitem{tomo}
 Pisanski T. 
 {\em Personal communication}

\bibitem{edge-trans}
\v{S}ir\'a\v{n} J., Tucker T. W., Watkins M.E., 
{\em Realizing finite edge-transitive orientable maps}, 
Journal of Graph Theory {\bf 37} (2001), 1--34.


\bibitem{mani} 
Wilson S.
 {\em Maniplexes: Part 1: Maps, Polytopes, Symmetry and Operators}, 
Symmetry, {\bf 4} (2012), 265--275

\end{thebibliography}
\end{document}